\documentclass[12pt]{amsart}  

\usepackage{graphicx}
\usepackage{here} 
\usepackage{array} 

\usepackage{vmargin}
\usepackage{amssymb}
\usepackage{mathrsfs}
\usepackage[all]{xy}
\usepackage[usenames,dvipsnames]{color}
\usepackage{soul}
\RequirePackage[colorlinks,linkcolor=blue,citecolor=LimeGreen,urlcolor=red]{hyperref} 
\usepackage{amsmath}

\usepackage{enumitem}

\newtheorem{theorem}{Theorem}[section]

\newtheorem{cor}[theorem]{Corollary}

\newtheorem{lemma}[theorem]{Lemma}

\newtheorem{prop}[theorem]{Proposition}
\newtheorem{remark}[theorem]{Remark}

\numberwithin{equation}{section}

\newcommand{\R}{\mathbb{R}}

\newcommand{\T}{\mathbb{T}}
\renewcommand{\S}{{\mathbb S}}

\newcommand{\func}[3]{#1 : #2 \longrightarrow #3}

\newcommand{\disp}{\displaystyle}
\newcommand{\abs}[1]{\left|#1\right|}
\newcommand{\eps}{\varepsilon}
\newcommand{\norm}[1]{\left\|#1\right\|}
\newcommand{\important}[1]{\textcolor{red}{#1}}
\renewcommand{\leq}{\leqslant}
\renewcommand{\geq}{\geqslant}
\renewcommand{\bar}{\overline}
\renewcommand{\tilde}{\widetilde}
\newcommand{\pa}[1]{\left(#1\right)}
\newcommand{\cro}[1]{\left[#1\right]}
\newcommand{\br}[1]{\left\{#1\right\}}
\newcommand\restr[2]{{
  \left.\kern-\nulldelimiterspace 
  #1 
  \right|_{ #2} 
  }}

\newcommand{\ie}{\textit{i.e. }}
\newcommand{\ini}{ { (\mbox{\scriptsize{in}}) }  }
\newcommand{\A}{ \mathbf{A}  }
\newcommand{\pl}{\partial_\ell}
\newcommand{\Ker}{\mathop{\mathrm{Ker}}}
\newcommand{\Image}{\mathop{\mathrm{Im}}}
\newcommand{\Maxwglob}{\mu}
\newcommand{\Maxwni}{M}

\newcommand{\vertiii}[1]{{\left\vert\kern-0.2ex\left\vert\kern-0.2ex\left\vert #1 
    \right\vert\kern-0.2ex\right\vert\kern-0.2ex\right\vert}}

\makeatletter
\def\namedlabel#1#2{\begingroup
    #2%
    \def\@currentlabel{#2}%
    \phantomsection\label{#1}\endgroup
}
\makeatother

\def\signmarc{\bigskip \begin{center} {\sc
Marc Briant\par
Universit\'e de Paris,\par 
MAP5, CNRS, \par 
F-75006 Paris, France \par

e-mail:} \tt{briant.maths@gmail.com} \end{center}}

\def\signberenice{\bigskip \begin{center} {\sc
B\'er\'enice Grec\par
Universit\'e de Paris,\par 
MAP5, CNRS, \par 
F-75006 Paris, France \par
e-mail:} \tt{berenice.grec@parisdescartes.fr} \end{center}}

\begin{document} 

\title[From the multi-species Boltzmann equation to the Fick system]{Rigorous derivation of the Fick cross-diffusion system from the multi-species Boltzmann equation in the diffusive scaling}
\author{Marc Briant, B\'er\'enice Grec}

\begin{abstract}
We present the arising of the Fick cross-diffusion system of equations for fluid mixtures from the multi-species Boltzmann in a rigorous manner in Sobolev spaces. To this end, we formally show that, in a diffusive scaling, the hydrodynamical limit of the kinetic system is the {Fick model supplemented with a closure relation} and we give explicit formulae for the macroscopic diffusion coefficients from the Boltzmann collision operator. Then, we provide a perturbative Cauchy theory in Sobolev spaces for the constructed Fick system, {which turns out to be a dilated parabolic equation.}  We finally prove the stability of the system in the Boltzmann equation, ensuring a rigorous derivation between the two models.
\end{abstract}

\maketitle

\vspace*{10mm}



\tableofcontents

\section{Introduction} \label{sec:intro}

The derivation and the mathematical analysis of models describing gaseous mixtures at different scales is a very active field in the literature. Such models are indeed widely used for different applications in physics and medicine, for example to describe the air flow in the distal part of the lungs, or  to model polluting particles in the atmosphere. 
In order to describe a dilute gaseous mixture of $N$ species, the mesoscopic sale is appropriate, representing species $i$ of the mixture by its distribution function $F_i(t,x,v)$, depending on time $t\in\R^+$, position in the $d$-dimensional torus $x\in\T^d$ and velocity $v\in\R^d$.
Several kinetic models have been introduced for mixtures \cite{BPS}, and we will here follow \cite{DMS} and consider that each function $F_i$ satisfies a Boltzmann-like equation, involving mono-species and bi-species collision operators $Q_{ii}$ and $Q_{ij}$, for any $1\leq i\neq j\leq N$.
As in the case of a mono-species gas \cite{GSB,CIP,AAP}, a H-theorem can be established in the multi-species case \cite{DMS} in the isothermal setting, proving that the equilibrium states of the collision operators are exactly Maxwellian distributions, with macroscopic observable quantities being the concentration of each species $n_i(t,x)$, and a bulk velocity $u(t,x)$.

We shall here focus on the diffusive fluid regime, meaning that both the Knudsen number, representing the average number of collisions undergone by a particle in a unit time, and the Mach number are small, taken to be equal to $\eps>0$ in our study. This diffusive scaling leads to the following rescaled multi-species Boltzmann equation for the distribution functions $F_i^\eps$
\begin{equation}\label{eq:multiBE}
 \eps \partial_t F_i^\eps + v \cdot \nabla_x F_i^\eps = \frac1\eps \sum_{j=1}^N Q_{ij} (F_i^\eps,F_j^\eps), \qquad  1\leq i \leq N. 
\end{equation}

A natural question is then to derive, formally and rigorously, a limit model of \eqref{eq:multiBE} when $\eps$ tends to zero. This has first been done formally, in the context of mixtures, both in the isothermal \cite{BGS,BGP1} and in the non-isothermal setting \cite{HutSal17}. Let us mention that other scalings can be considered, in particular not assuming the Mach number to be of order $\eps$, leading to the Euler or the Navier-Stokes limit, which have been intensively studied, both formally and rigorously, in the context of mono-species gases (see for example \cite{BGL1,BGL2,GSR}), and partially for mixtures \cite{BisMarSpi,BisDes,BD,BriDau,BBBD}. 

An important question that rises during the formal derivation of a limit model for \eqref{eq:multiBE} when $\eps$ tends to zero is the assumption made on the macroscopic velocities of each species. Indeed, as we mentioned earlier, the equilibria of the multi-species Boltzmann collision operator are Maxwellian distributions with a mutual bulk velocity to all species. However, in a rarefied regime, a natural assumption is to consider that each species moves with its own velocity, independently of the possible diffusive scaling chosen for the equation. Focusing henceforth on the isothermal setting, this is what is done in the previous works \cite{BGS,HS1,BGP1}, using for example a moment method with the ansatz that the distribution functions of each species are local Maxwellian distributions with different velocities (of order $\eps$) for each species. This setting leads to the Maxwell-Stefan equations, where the gradient of the concentration $\nabla_x n_i$ of each species is expressed through the so-called Maxwell-Stefan matrix as a function of the macroscopic flux $(n_i u_i)_{1\leq i \leq N}$ of all species. 

Another approach is to consider that at the leading order, the species velocities are all the same, which is the case when writing an Hilbert or a Chapman-Enskog expansion of each distribution function $F_i^\eps$ around an equilibrium of the collision operator, that is a Maxwellian distribution with concentrations $n_i$ for each species, and a bulk velocity. This is the point of view we chose in this paper. As we shall see in the following, the limit system obtained in this case is the Fick one, in which the macroscopic flux $J_i$ of each species, defined as the moment of order $1$ in velocity of the correction at order $\eps$, is related to the concentration gradients of all species through the so-called Fick matrix. The Fick system, which expresses the fluxes as functions of the concentration gradients, can be combined with the mass conservation equations to get rid of the fluxes and obtain a standard cross-diffusion equation \cite{DesLepMou,DesLepMouTre}. 

Despite their structural similarities, the Fick and Maxwell-Stefan systems are not equivalent, since the two involved matrices are not invertible. Of course, using a closure relation (for example equimolar diffusion setting \cite{Bothe,BGS0,JS}) or using pseudo-inversion \cite{Gio}, the two models can be linked. However, as we already stated, they are not constructed using the same assumptions concerning the species velocities. It has been proved in \cite{BGP2} that the Fick model can be seen as the limit of the Maxwell-Stefan one in the stiff limit when all species velocities are equal (even in the non-diffusive setting).

 A natural issue is then to justify rigorously the formal convergence of the multi-species Boltzmann equation towards these macroscopic diffusion systems (Fick or Maxwell-Stefan). This falls into the wide literature concerning the hydrodynamical limits of kinetic equations \cite{SR}. In the context of mixtures, {it has been proved in \cite{BonBri} that the Maxwell-Stefan model is stable for the Boltzmann multi-species equation, ensuring a rigorous derivation of the Maxwell-Stefan system in a perturbative setting.} In their paper, the authors choose to consider perturbative solutions around a Maxwellian whose fluid quantities solve the limit macroscopic system as in \cite{Caflisch,DMEL}, and use hypocoercive strategy in the spirit of \cite{MouNeu,Guo,Briant}.
 
In this paper, we shall tackle the rigorous limit towards the Fick model in a perturbative setting as well, following the same ideas as in \cite{BonBri}. More precisely, we first derive formally the Fick diffusion coefficients, and show that they are naturally linked to the inverse of the linear Boltzmann operator. Next, we develop a Cauchy theory for the Fick system in a perturbative setting, which is inherent to the hydrodynamic limits in a dissipative regime. {The Fick equation arising from the Boltzmann equation is degenerate but unlike the Maxwell-Stefan model it is not parabolic, due to the lack of symmetry of the Fick matrix.
Standard parabolic approaches fail in this context (see Remark \ref{rem:dilated}). We shall exhibit a dilated parabolicity and solve it by intertwining a time and space rescaling. 
Lastly, the convergence between the mesoscopic and the macroscopic model is proved} by showing that the Maxwellian whose concentrations satisfy the perturbed Fick system is a stable state of order $\eps$ for the Boltzmann system.
 
The outline of the paper is the following. First, we describe in Section \ref{sec:kinetic setting} the kinetic multi-species setting, and state our main results. Then, in Section \ref{sec:inverse operator}, we give some properties of the inverse of the linear Boltzmann operator, and in particular a spectral gap property for this operator, giving explicit expressions of the constants (and specifically keeping track of their dependencies on the concentrations). After deriving formally the Fick system from the Boltzmann one in Section \ref{subsec:Fick matrix}, we will prove some properties of the Fick matrix in Section \ref{subsec:properties Fick matrix}. We will then be able to prove a perturbative Cauchy theory for the Fick equation in Section \ref{sec:resolution Fick}, which will allow to conclude the rigorous convergence in Section \ref{sec:BE} thanks to a result established in \cite{BonBri}.
\bigskip

\section{Kinetic setting and statement of the main results}\label{sec:kinetic setting}

\subsection{Kinetic description of the mixture}

The mixture is considered to be a dilute gas composed of $N$ different species of chemically non-reacting mono-atomic particles. In order to avoid any confusion, vectors and vector-valued operators in $\R^N$ will be denoted by a bold symbol, whereas their components by the same indexed symbol. For instance, $\mathbf{W}$ represents the vector or vector-valued operator $(W_1,\dots,W_N)$. 
The multispecies Boltzmann operator is a vector-valued operator $\mathbf{Q}(\mathbf{F})=(Q_1(\mathbf{F}),\allowbreak\ldots,\allowbreak Q_N(\mathbf{F}))$  acting only on the velocity variable. For a vector-valued function $\mathbf{F}(v) = \pa{F_i(v)}_{1\leq i \leq N}$, the former operator is given for all $1 \leq i\leq N$ by 
$$Q_i(\mathbf{F}) = \sum\limits_{j=1}^N Q_{ij}(F_i,F_j),$$
where $Q_{ij}$ describes interactions between particles of either the same ($i=j$) or different ($i\neq j$) species, which are local in time and space. It is given by
$$Q_{ij}(F_i,F_j)(v) =\int_{\R^d\times \mathbb{S}^{d-1}}B_{ij}\left(|v - v_*|,\mbox{cos}\:\theta\right)\left[F_i'F_j^{'*} - F_iF_j^*\right]dv_*d\sigma, \qquad 1\leq i,j\leq N,$$
where we used the shorthands $F_i'=F_i(v')$, $F_i=F_i(v)$, $F_j^{'*}=F_j(v'_*)$ and $F_j^*=F_j(v_*)$, with the pre-collisional velocities $v'$ and $v'_*$ defined by
$$\left\{ \begin{array}{rl} \displaystyle{v'} & \displaystyle{=\frac{1}{m_i+m_j}\pa{m_iv+m_jv_* +  m_j|v-v_*|\sigma}} \vspace{2mm} \\ \vspace{2mm} \displaystyle{v' _*}&\displaystyle{=\frac{1}{m_i+m_j}\pa{m_iv+m_jv_* -m_i  |v-v_*|\sigma}} \end{array}\right.,$$
and $\cos \theta =  \frac{(v-v_*)\cdot \sigma}{\abs{v-v_*}}$. The masses of species $i$ and $j$ are denoted respectively by $m_i>0$ and $m_j>0$.
Note that these expressions imply that we deal with gases where only binary elastic collisions occur (the mass $m_i$ of all molecules of species $i$ remains the same, since there is no reaction). More precisely, $v'$ and $v'_*$ are the velocities of two molecules of species $i$ and $j$ before collision, which give post-collisional velocities $v$ and $v_*$ respectively, with conservation of momentum and kinetic energy:
\begin{equation}\label{elasticcollision}
\begin{split}
m_iv + m_jv_* &= m_iv' + m_jv'_*,
\\\frac{1}{2}m_i\abs{v}^2 + \frac{1}{2}m_j\abs{v_*}^2 &= \frac{1}{2}m_i\abs{v'}^2 + \frac{1}{2}m_j\abs{v'_*}^2.
\end{split}
\end{equation}

For the collision kernels, we assume that they satisfy standard assumptions stated below in the multi-species setting \cite{DauJunMouZam,BriDau}, which are also standard in  the mono-species case \cite{BarMou,Mou1} to obtain spectral properties for the linear operator. 
\begin{enumerate}[label=(H\arabic*)]
\item The following symmetry property holds
$$B_{ij}(|v-v_*|,\cos\theta) = B_{ji}(|v-v_*|,\cos\theta),\qquad 1\leq i,j\leq N.$$
This assumption conveys the idea that collisions are micro-reversible.
\item The collision kernels decompose into the product of a kinetic part $\Phi_{ij}\ge 0$ and an angular part $b_{ij}\ge 0$
$$ B_{ij}(|v-v_*|,\cos\theta) = \Phi_{ij}(|v-v_*|)b_{ij}(\cos\theta),\qquad 1\leq i,j\leq N.$$
This assumption is used for the sake of clarity but covers a wide range of physical applications.
\item The kinetic part has the form of hard or Maxwellian ($\gamma=0$) potentials, \textit{i.e.} there exist $C_{ij}^{\Phi}>0,~\:\gamma\in[0,1]$ such that
$$\Phi_{ij}(|v-v_*|)=C_{ij}^{\Phi}|v-v_*|^{\gamma}, \qquad 1 \leq i,j\leq N.$$
It holds for collision kernels coming from interaction potentials which behave like power-laws. 
\item For the angular part, we assume a strong form of Grad's angular cutoff \cite{Gr1}, namely that there exist constants $C_{b1}$, $C_{b2}>0$ such that,
for all $1\le i,j\le N$ and $\theta\in[0,\pi]$,
$$  0<b_{ij}(\cos\theta)\le C_{b1}|\sin\theta|\,|\cos\theta|, \quad b'_{ij}(\cos\theta)\le C_{b2}.$$
Furthermore, 
$$  C^b := \min_{1\le i\le N}\inf_{\sigma_1,\sigma_2\in\S^{d-1}}\int_{\S^{d-1}}\min\big\{	b_{ii}(\sigma_1\cdot\sigma_3),b_{ii}(\sigma_2\cdot\sigma_3)\big\}\:d\sigma_3 > 0. $$
This positivity assumption is satisfied by most physical models and is required to obtain an explicit spectral gap in the mono-species case \cite{BarMou,Mou1} and is thus a prerequisite for having a spectral gap in the multi-species case \cite{BriDau} (the boundedness of $b_{ij}'$ could however be relaxed but in that case the spectral gap is not explicit \cite{DauJunMouZam}).
\end{enumerate}

Using the standard changes of variables $(v,v_*) \mapsto (v',v'_*)$ and $(v,v_*) \mapsto (v_*,v)$ together with the symmetries of the collision operators (see \cite{Ce,CIP,Vi2} among others and \cite{DMS,BGS,DauJunMouZam} in the multi-species setting), we recover the following weak forms for any $1\leq i,j\leq N$ and any test functions $\psi_i,\psi_j$ such that the following expressions make sense
$$\int_{\R^d}Q_{ij}(F_i,F_j)(v)\psi_i(v)\:dv =\int_{\R^{2d}\times\S^{d-1}}B_{ij}(|v-v_*|,\cos\theta)F_iF_j^* \left(\psi_i'-\psi_i\right)\:d\sigma dvdv_*,$$
and
\begin{equation}\label{symmetry property Qij}
\begin{split}
& \int_{\R^d}Q_{ij}(F_i,F_j)(v)\psi_i(v)\:dv + \int_{\R^d}Q_{ji}(F_j,F_i)(v)\psi_j(v)\:dv=\\
&\quad - \frac{1}{2}\int_{\R^{2d}\times\S^{d-1}}B_{ij}(|v-v_*|,\cos\theta)\left(F_i'F_j^* - F_iF_j^*\right)\left(\psi_i' + \psi_j'^* - \psi_i - \psi_j^*\right)\:d\sigma dvdv_*.
\end{split}
\end{equation}

Thus, the relation
\begin{equation}\label{invariantsQij}
\sum\limits_{i,j=1}^N\int_{\R^d} Q_{ij}(F_i,F_j)(v)\psi_i(v)\:dv =0
\end{equation}
holds if and only if $\boldsymbol{\psi}(v)$ belongs to $\mbox{Span}\br{\mathbf{e_1},\dots,\mathbf{e_N},v_1\mathbf{m},v_2\mathbf{m},v_3\mathbf{m},\abs{v}^2\mathbf{m}}$, where $\mathbf{e_k}$ stands for the $k^{th}$ unit vector in $\R^N$ and $\mathbf{m} = (m_1,\dots,m_N)$.
\par The operator $\mathbf{Q}=(Q_1,\dots,Q_N)$ also satisfies a multi-species version of the classical H-theorem \cite{DMS} which implies that any local equilibrium, i.e. any function $\mathbf{F}=(F_1,\dots,F_N)$ being the maximum of the Boltzmann entropy, has the form of a local Maxwellian, meaning that there exist functions $n_{\mbox{\scriptsize{loc}},i}$, $1\leq i \leq N$, $u_{\mbox{\scriptsize{loc}}}$ and $\theta_{\mbox{\scriptsize{loc}}}$ depending on $t,x$ such that
$$\forall\:1\leq i \leq N,\:\: F_i(t,x,v) = n_{\mbox{\scriptsize{loc}},i}(t,x)\pa{\frac{m_i}{2\pi k_B \theta_{\mbox{\scriptsize{loc}}}(t,x)}}^{d/2}\mbox{exp}\cro{-m_i\frac{\abs{v-u_{\mbox{\scriptsize{loc}}}(t,x)}^2}{2k_B\theta_{\mbox{\scriptsize{loc}}}(t,x)}},$$
where $k_B$ is the Boltzmann constant.
\par For each species we associate a local equilibrium $\Maxwni_i(t,x,v)$ that is related to the multi-species Boltzmann operator (see \cite{BGPS,BGS}), chosen with zero bulk velocity and temperature equal to $1$ for simplicity. It is given, for any $1\leq i \leq N$, by
\begin{equation}\label{eq:Mi}
 \Maxwni_i(t,x,v) = n_i(t,x)\pa{\frac{m_i}{2\pi}}^{d/2}e^{-m_i\frac{\abs{v}^2}{2}}, \qquad \forall (t,x,v) \in \R^+\times \T^d \times \R^d,
\end{equation}
where the concentration of each species in the fluid is denoted by $n_i(t,x)$.
In the sequel, we shall use the notation $\boldsymbol\Maxwni = \pa{\Maxwni_i}_{1\leq i \leq N}$. Introducing the global Maxwellian $\boldsymbol\Maxwglob = (\Maxwglob_i)_{1 \leq i \leq N}$, defined by 
\begin{equation}\label{eq:mui}
\Maxwglob_i (v)= \pa{\frac{m_i}{2\pi}}^{d/2}e^{-m_i\frac{\abs{v}^2}{2},}
\end{equation}
observe that $\Maxwni_i (t,x,v) = n_i(t,x) \Maxwglob_i(v)$.

\par We can associate to $\boldsymbol\Maxwni$ a linearisation of the Boltzmann operator, namely $\mathbf{L}\pa{\mathbf{f}} = \pa{L_i\pa{\mathbf{f}}}_{1\leq i \leq N}$, where
\begin{equation}\label{L}
L_i(\mathbf{f}) = \sum_{j=1}^N L_{ij}(\mathbf{f}), \quad 1\le i\le N,
\end{equation}
with
\begin{equation}\label{Lij}
\begin{split}
  L_{ij}(\mathbf{f}) &=Q_{ij}\left(\Maxwni_i,f_j\right) 	+ Q_{ij}\left(f_i,\Maxwni_j\right)  \\
	&= \int_{\R^d\times\S^{d-1}}B_{ij}(|v-v_*|, \cos \theta)\left(\Maxwni_j'^*f_i' + \Maxwni_i'f_j'^* - \Maxwni_j^*f_i - \Maxwni_i f_j^* \right)\:dv_*d\sigma.
\end{split}
\end{equation}

The operator $\mathbf{L}$ can be written under the form $\mathbf{L} = -\mathbf{\nu}(v) + \mathbf{K}$, where $\mathbf{K}$ is a compact operator and $\mathbf{\nu} = (\nu_i)_{1\leq i \leq N}$ is the collision frequency, with $\nu_i(v) = \sum_{j=1}^N \nu_{ij} (v)$, and
\begin{equation}\label{nuij}
 \nu_{ij} (v) = C_{ij}^\Phi \int_{\R^d\times \S^{d-1}} b_{ij} (\cos\theta) |v-v_*|^\gamma \Maxwni_j(v_*) \:dv_*d\sigma.
\end{equation}

\subsection{Main results}\label{sec:main results}

In order to state our main results, let us define some notations.
We define the Euclidian scalar product in $\R^N$ weighted by a vector $\mathbf{W}$ by
$$\langle \mathbf{f},\mathbf{g}\rangle_{\mathbf{W}}= \sum\limits_{i=1}^Nf_ig_i W_i.$$
In the case $\mathbf{W}=\mathbf{1}=(1,\dots,1)$ we may omit the index $\mathbf{1}$.
For function spaces, we index the space by the name of the concerned variable, so that, for $p$ in $[1,+\infty]$
$$L^p_{[0,T]} = L^p\pa{[0,T]},\quad L^p_{t} = L^p \left(\R^+\right),\quad L^p_x = L^p\left(\T^d\right), \quad L^p_v = L^p\left(\R^d\right).$$
\par For $\func{\mathbf{W}=(W_1, \ldots, W_N)}{\R^d}{\R^+}$ a strictly positive measurable function in $v$, we will use the following vector-valued weighted Lebesgue spaces defined by their norms 
\begin{equation*}\label{norm}
     \begin{array}{ll}
 \norm{f}_{L^p_{v}\pa{\mathbf{W}}} = \left(\sum\limits_{i=1}^N \norm{f_i}^2_{L^p_{v}(W_i)}\right)^{1/2},     &    \norm{f_i}_{L^p_{v}(W_i)}=\left\|f_i W_i(v)\right\|_{L^p_v},\\
 \norm{f}_{L^p_{x,v}\pa{\mathbf{W}}} = \left(\sum\limits_{i=1}^N \norm{f_i}^2_{L^p_{x,v}(W_i)}\right)^{1/2}, &    \norm{f_i}_{L^p_{x,v}(W_i)}=\left\|\|f_i\|_{L^p_x}W_i(v)\right\|_{L^p_v},\\
 \norm{f}_{L^{\infty}_{x,v}\pa{\mathbf{W}}} = \sum\limits_{i=1}^N \norm{f_i}_{L^{\infty}_{x,v}(W_i)},        &    \norm{f_i}_{L^{\infty}_{x,v}\pa{W_i}} = \sup\limits_{(x,v)\in \T^d \times \R^d}\big(\left|f_i(x,v)\right|W_i(v)\big).
        \end{array}
 \end{equation*}
Note that $L^2_v(\mathbf{W})$ and $L^2_{x,v}(\mathbf{W})$ are Hilbert spaces with respect to the scalar products 
\begin{eqnarray*}
\langle \mathbf{f},\mathbf{g}\rangle_{L^2_{v}(\mathbf{W})}&=&\sum\limits_{i=1}^N \langle f_i,g_i \rangle_{L^2_{v}(W_i)}=\sum\limits_{i=1}^N \int_{\R^d} f_i g_i W_i^2 dv,
\\\langle\mathbf{f},\mathbf{g}\rangle_{L^2_{x,v}(\mathbf{W})}&=&\sum\limits_{i=1}^N \langle f_i,g_i \rangle_{L^2_{x,v}(W_i)}=\sum\limits_{i=1}^N \int_{\T^d \times \R^d} f_i g_i W_i^2 dxdv.
\end{eqnarray*}

One can construct a Fick cross-diffusion matrix $\A(\mathbf n)$ from the Boltzmann collision operator, see Section \ref{subsec:Fick matrix} for the construction and an explicit formula \eqref{eq:FickMatrix}.  One expects some perturbative solution to the multispecies Boltzmann equation to converge to the Fick system in the sense of their hydrodynamic quantities.  Provided one can construct perturbative solutions to the Fick system of the form $\mathbf{n}=\mathbf{n_\infty}+\eps\mathbf{\tilde{n}}$, the next theorem states that the Fick Maxwellian
\begin{equation}\label{eq:Fick maxwellian}
\forall (t,x,v) \in \R^+\times\T^d\times\R^d,\quad \mathbf{M}^{\boldsymbol\eps}(t,x,v) = \pa{\mathbf{n_\infty}+\eps\mathbf{\tilde{n}}(t,x)}\boldsymbol\mu(v)
\end{equation}
is a stable state of order $\eps$ of the Boltzmann system, which shows the hydrodynamic limit from Boltzmann multispecies to the Fick system.

\begin{theorem}\label{theo:BE}
Let $\mathbf{n^\eps}(t,x)=\mathbf{n_\infty}+\mathbf{\tilde{n}}(t,x)$ be a perturbative solution of the Fick system (constructed in Theorem \ref{theo:Fick}), which defines the Maxwellian $\mathbf M^\eps(t,x,v) = \mathbf{n^\eps}(t,x) \boldsymbol\mu(v)$. Assume $(H1)-(H2)-(H3)-(H4)$ are satisfied on the collision kernel and that $s>d/2$.
There exist real numbers $\delta_{\rm fluid},\:\delta_B >0$ such that, if the initial datum $\mathbf{F^{(\mbox{\footnotesize in})}}$ satisfies 
\begin{itemize}
\item[(i)]  $\norm{\mathbf{\tilde{n}^{\rm (in)}}}_{H^{s+2}_x} \leq \delta_{\rm fluid}$,
\item[(ii)]$\disp{\mathbf{f^{(\mbox{\footnotesize{in}})}}} \in \mathcal{H}^s_\eps$ with $\norm{\mathbf{f^{(\mbox{\footnotesize{in}})}}}_{\mathcal{H}^s_\eps} \leq \delta_B$ and $\disp{\abs{\int_{\T^d}\boldsymbol\pi_{\mathbf{L}}(\mathbf{f^{(\mbox{\footnotesize{in}})}})dx} \leq  \delta_{\mbox{\footnotesize{fluid}}}}$, where $\boldsymbol\pi_{\mathbf{L}}$ is the orthogonal projection in $L^2_v(\boldsymbol\mu^{-1/2})$ onto $\Ker(\mathbf{L})$ (see Subsection \ref{subsec:properties BE}),
\end{itemize}
then the multispecies Boltzmann equation \eqref{eq:multiBE} possesses a unique global perturbative solution $\mathbf{F^\eps}(t,x,v) = \mathbf{M}^{\boldsymbol\eps}(t,x) + \eps \mathbf{f^\eps}(t,x,v) \geq 0$, with $\mathbf{f^\eps}\in C^0\pa{\R^+;H^s_{x,v}\pa{\boldsymbol\mu^{-1/2}}}$.\\ 
Moreover, there exists a constant $C_B>0$ and a norm $\norm{\cdot}_{\mathcal{H}^s_\eps} $, equivalent to the following weighted hypocoercive norm
\begin{equation*}
\norm{\cdot}_{\mathcal{H}^s_\eps}^2 \sim  \norm{\cdot}_{L^2_{x,v}\pa{\boldsymbol\mu^{-1/2}}}^2 + \sum_{|\alpha| \leq s}\norm{\partial^\alpha_x\cdot}_{L^2_{x,v}\pa{\boldsymbol\mu^{-1/2}}}^2 + \eps^2\sum_{\substack{|\alpha|+|\beta|\leq s   \\[0.2mm]  |\beta|\geq 1}}\norm{\partial^\beta_v\partial^\alpha_x\cdot}_{L^2_{x,v}\pa{\boldsymbol\mu^{-1/2}}}^2,
\end{equation*}
such that the solution to the Boltzmann equation \eqref{eq:multiBE} satisfies the following stability property for all $t\geq0$
$$ \norm{\mathbf{F^\eps}-\mathbf{M}^{\boldsymbol\eps}}_{\mathcal{H}^s_\eps}(t) \leq \eps C_B. $$
All the constant are explicit and independent of $\eps$.
\end{theorem}

\begin{remark}
 It is important to note that the $\mathcal{H}^s_\eps$-norm does not display any $\eps$-factors in front of the norms of pure spatial derivatives. As the hydrodynamical limit only concerns integration over the velocity variable it means that we indeed have a strong convergence of $\eps \int_{\R^d}\mathbf{f^\eps}(t,x,v)dv$ towards $0$ in $H^s_x$ as $\eps$ vanishes.

Moreover, we loose $2$ steps of regularity between the fluid solutions $\mathbf{n}$ and the solutions of the Boltzmann equation in Theorem \ref{theo:BE}. However, as we shall detail in Remark \ref{rem:Hs+1}, Theorem \ref{theo:BE} could be rewritten with $\mathbf{\tilde{n}^{\rm (in)}}$ solely in $H^{s+1}_x$, and proved with the same methodology, but working in $L^2_t\mathcal{H}^s_\eps$ rather that $L^\infty_t\mathcal{H}^s_\eps$.
\end{remark}

As we stated it before, the previous stability result relies on the construction of a perturbative Cauchy theory around a stationary state for the associated Fick equation
\begin{equation}\label{eq:Fick intro}
 \left\{\begin{array}{l}\disp{\partial_t \mathbf n + \nabla_x\cdot \pa{\A(\mathbf n) \nabla_x \mathbf n} =0}, \\ 
\disp{\sum\limits_{i=1}^Nm_in_i(t,x) = \sum\limits_{i=1}^Nm_in_{\infty,i} .}\end{array}\right. 
\end{equation}
 This is done in the following theorem.

\begin{theorem}\label{theo:Fick}
Let $\func{\A}{\R^N}{M_{dN,d}(\R)}$ be the Fick matrix (defined in Section \ref{subsec:Fick matrix}). 
Let $s>d/2$ be an integer, let $\delta>0$ and $\mathbf{n_\infty} >0$. 
There exist $\delta_s >0$ and $\lambda_s >0$ such that, if the initial datum $\mathbf{\tilde{n}^\ini}$ satisfies
 \begin{enumerate}[label=(\roman*)]
\item $\disp{\forall x \in\T^d,\:\mathbf{n_\infty} + \mathbf{\tilde{n}^\ini}(x) \geq \delta}$ and $\disp{\int_{\T^d}\mathbf{\tilde{n}^\ini}(x)dx=0}$,
\item $\disp{\forall x \in \T^d, \: \sum\limits_{i=1}^N m_i \tilde{n}^\ini_i(x) = 0}$,
\item $\disp{\norm{\mathbf{\tilde{n}^\ini}}_{H^s_x} \leq \delta_s}$,
\end{enumerate}
then there exists a unique solution $\mathbf{n}(t,x) = \mathbf{n_\infty} + \mathbf{\tilde{n}}(t,x)$ on $\R^+$ to the Fick equation \eqref{eq:Fick intro}. Moreover, it satisfies, for any $t \geq 0$
\begin{enumerate}[label=(\alph*)]
\item $\disp{\forall x \in \T^d,\:\mathbf{n_\infty} + \mathbf{\tilde{n}}(t,x) \geq \delta }$ and $\disp{\int_{\T^d}\mathbf{\tilde{n}}(t,x)dx=0}$;
\item $\disp{\norm{\mathbf{\tilde{n}}(t)}_{H^s_x}\leq  \norm{\mathbf{\tilde{n}^{(\mbox{\footnotesize{in})}}}}_{H^s_x}}e^{-\lambda_s t}$.
\end{enumerate}
The constants $\delta_s$ and $\lambda_s$ only depend on $s$ and $\delta$.
\end{theorem}

\begin{remark}
Observe that imposing a mean-free property on $\mathbf{\tilde{n}^\ini}$ is not necessary and is only used for convenience. In the case of a non-zero mean initial perturbation the same arguments would just apply with initial datum
$$\mathbf{n_\infty}+\mathbf{\tilde{n}^{\ini}} -\int_{\T^d}\mathbf{\tilde{n}^\ini}(x)dx.$$
Moreover, the uniqueness is only to be understood in a perturbative sense, which means among the solution of the form $\mathbf{n_\infty}+\mathbf{\tilde{n}}$ where the perturbation $\mathbf{\tilde{n}}$ remains small in $H^s_x$.
\end{remark}

\section{Properties of the inverse of the linear Boltzmann operator}\label{sec:inverse operator}

The Fick matrix will involve the inverse of the multispecies linear Boltzmann  operator. Let us therefore first describe how the latter is defined and obtain explicit bounds, depending on the concentration of each species $n_i(t,x)$. Because the linear Boltzmann operator only acts on the velocity variable, the results stated in this section are local in $(t,x)$ and for the sake of readability, we do not write down the $(t,x)$-dependences.


\subsection{Well-posedness, boundedness and spectral gap}\label{subsec:properties BE}

We start with a description of some well-known properties \cite{BGPS,BGS,BriDau} of the multi-species Boltzmann operator. We recall the definition \eqref{L}-\eqref{Lij} of $\mathbf{L}$. 
Of course, the case of exactly $N$ species only makes sense if all the $n_i$ are positive. Indeed, if one $n_i$ is zero, then we only have $N-1$ species and the following holds with $N$ replaced by $N-1$. We thus assume in the following that $\min\limits_{1\leq i \leq N}\br{n_i} >0$.
\par From \cite{DauJunMouZam,BriDau}, $\mathbf{L}$ is a self-adjoint operator in $L^{2}_v(\boldsymbol\Maxwni^{-1/2})$ with $\langle \mathbf{f}, \mathbf{L}(\mathbf{f})\rangle_{L^2_v(\boldsymbol\Maxwni^{-1/2})}=0$ if and only if $\mathbf{f}$ belongs to $\Ker(\mathbf{L})$, where $\Ker\left(\mathbf{L}\right)$ is spanned by the functions $\boldsymbol\phi_i$, $1\leq i \leq N+d+1$,
with
\begin{equation}\label{piL}
\left\{\begin{array}{l} \disp{\boldsymbol\phi_k(v)=\frac{1}{\sqrt{n_{k}}}\:\Maxwni_k\mathbf{e_k},\quad 1\leq k \leq N},
\\\vspace{2mm}  \disp{ \boldsymbol\phi_k(v) =\frac{v_{k-N}}{\pa{\sum\limits_{i=1}^N m_in_{i}}^{1/2}}\:\pa{m_i\Maxwni_i}_{1\leq i \leq N},\quad N+1\leq k \leq N+d},
\vspace{2mm}\\\vspace{2mm} \disp{\boldsymbol\phi_{N+d+1}(v)=\frac{1}{\pa{\sum\limits_{i=1}^N n_{i}}^{1/2}}\:\pa{\frac{\abs{v}^2-dm_i^{-1}}{\sqrt{2d}}m_i\Maxwni_i}_{1\leq i \leq N}.}\end{array}\right.
\end{equation}
with the notation $\mathbf{e_k} = \pa{\delta_{ik}}_{1\leq i \leq N}$.
These functions  $\pa{\boldsymbol\phi_i}_{1\leq i\leq N+d+1}$ form an orthonormal basis of $\Ker(\mathbf{L})$ in $L^2_v(\boldsymbol\Maxwni^{-1/2})$. Let us denote $\pi_{\mathbf{L}}$ the orthogonal projection onto $\Ker(\mathbf{L})$ in $L^2_v(\boldsymbol\Maxwni^{-1/2})$
$$\pi_{\mathbf{L}}(\mathbf{f}) = \sum\limits_{k=1}^{N+d+1} \pa{\int_{\R^d} \langle\mathbf{f}(v),\boldsymbol\phi_k(v)\rangle_{L^2_v(\boldsymbol\Maxwni^{-1/2})}\:dv} \boldsymbol\phi_k(v).$$
\par An important property of the operator $\mathbf{L}$ is that it is non-positive. This translates into the following spectral gap result proved in \cite{BriDau}

\begin{prop}\label{prop: properties L}
The operator $\mathbf{L}$ is a closed self-adjoint operator in $L^2_v(\boldsymbol\Maxwni^{-1/2})$ and there exists $\lambda_L >0$ such that
$$\forall \mathbf{f} \in L^2_v(\boldsymbol\Maxwni^{-1/2}), \quad \left\langle \mathbf{f}, \mathbf{L}\pa{\mathbf{f}} \right\rangle_{L^2_v(\boldsymbol\Maxwni^{-1/2})} \leq -\lambda_L \norm{\mathbf{f} - \pi_\mathbf{L}\pa{\mathbf{f}}}^2_{L^2_v(\boldsymbol\Maxwni^{-1/2})},$$
and there exists $C_L>0$ such that
$$\forall \mathbf{f} \in L^2_v(\boldsymbol\Maxwni^{-1/2}), \quad \norm{\mathbf{L}(\mathbf{f})}_{L^2_v(\boldsymbol\Maxwni^{-1/2})} \leq C_L \norm{\mathbf{f}}_{L^2_v(\boldsymbol\Maxwni^{-1/2})}.$$
\end{prop}

Thanks to the above proposition we can define $\mathbf{L}^{-1}$ on $\Ker(\mathbf{L})^\bot = \mbox{Im}(\mathbf{L})$ and we have the following proposition on $\mathbf{L}^{-1}$.

\begin{prop}\label{prop: spectral gap L-1}
The operator $\mathbf{L}^{-1}$ is a self-adjoint operator in $\Ker(\mathbf{L})^\bot$ and for any $\mathbf{h}$ in $\pa{\Ker(\mathbf{L})}^\bot= \emph{\mbox{Im}}(\mathbf{L})$ the following holds
\begin{itemize}
\item[(i)] $\disp{\norm{L^{-1}(\mathbf{h})}_{L^2_v(\boldsymbol\Maxwni^{-1/2})} \leq \frac{1}{\lambda_L}\norm{\mathbf{h}}_{L^2_v(\boldsymbol\Maxwni^{-1/2})}}$;
\item[(ii)]$\disp{\left\langle \mathbf{h}, \mathbf{L}^{-1}\pa{\mathbf{h}} \right\rangle_{L^2_v(\boldsymbol\Maxwni^{-1/2})} \leq -\frac{\lambda_L}{C_L^2} \norm{\mathbf{h}}^2_{L^2_v(\boldsymbol\Maxwni^{-1/2})}}$;
\end{itemize}
where $\lambda_L$, $C_L >0$ have been defined in Proposition \ref{prop: properties L}.
\end{prop}

\begin{proof}[Proof of Proposition \ref{prop: spectral gap L-1}]
The proof is a direct application of the spectral gap property of $\mathbf{L}$ (Proposition \ref{prop: properties L}). Indeed, applying Cauchy-Schwarz inequality yields, for all $\mathbf{f}$ in $\Ker(\mathbf{L})^\bot$
$$-\norm{\mathbf{f}}_{L^2_v(\boldsymbol\Maxwni^{-1/2})}\norm{\mathbf{L}(\mathbf{f})}_{L^2_v(\boldsymbol\Maxwni^{-1/2})} \leq -\lambda_L\norm{\mathbf{f}}_{L^2_v(\boldsymbol\Maxwni^{-1/2})},$$
so that
$$\norm{\mathbf{f}}_{L^2_v(\boldsymbol\Maxwni^{-1/2})} \leq \frac{1}{\lambda_L}\norm{\mathbf{L}(\mathbf{f})}_{L^2_v(\boldsymbol\Maxwni^{-1/2})},$$
which is $(i)$ taking $\mathbf{f} = \mathbf{L}^{-1}(\mathbf{h})$.
\par The spectral gap property $(ii)$ comes first from the boundedness of $\mathbf{L}$ (Proposition~\ref{prop: properties L}) for $\mathbf{f} = \mathbf{L}^{-1}(\mathbf{h})$ which translates into a coercivity property of $\mathbf{L}$
$$\norm{\mathbf{h}}_{L^2_v(\boldsymbol\Maxwni^{-1/2})}^2 \leq C_L^2 \norm{\mathbf{L}^{-1}(\mathbf{h})}^2_{L^2_v(\boldsymbol\Maxwni^{-1/2})},$$
which we plug into the spectral gap inequality satisfied by $\mathbf{L}$.
\end{proof}

\begin{remark}
The spectral gap result on $\mathbf{L}$ actually holds in a more regular space, which therefore translates onto $\mathbf{L}^{-1}$. Defining the shorthand notation
$\langle v \rangle = \sqrt{1+\abs{v}^2}$, we have, for any $\mathbf{f} \in L^2_v(\boldsymbol\Maxwni^{-1/2})$
\begin{align*}
&\left\langle \mathbf{f}, \mathbf{L}\pa{\mathbf{f}} \right\rangle_{L^2_v(\boldsymbol\Maxwni^{-1/2})} \leq -\lambda_L \norm{\mathbf{f} - \pi_\mathbf{L}\pa{\mathbf{f}}}^2_{L^2_v\pa{\langle v \rangle^{\gamma/2}\boldsymbol\Maxwni^{-1/2}}},\\
 &\norm{\mathbf{L}(\mathbf{f})}_{L^2_v(\boldsymbol\Maxwni^{-1/2})} \leq C_L \norm{\mathbf{f}}_{L^2_v\pa{\langle v \rangle^{\gamma/2}\boldsymbol\Maxwni^{-1/2}}},
\end{align*}
 and, for any $\mathbf{h} \in\Ker(\mathbf{L})^\bot$,
 \[\norm{L^{-1}(\mathbf{h})}_{L^2_v\pa{\langle v \rangle^{\gamma/2}\boldsymbol\Maxwni^{-1/2}}} \leq \frac{1}{\lambda_L}\norm{\mathbf{h}}_{L^2_v(\boldsymbol\Maxwni^{-1/2})}.\]
\end{remark}


\subsection{Explicit dependencies on the concentrations}\label{subsec:explicit dependencies}

In order to derive estimates on $\pa{n_i(t,x)}_{1\leq i \leq N}$ for the Fick system it is of core importance to find out the dependencies of $\lambda_L$ and $C_L$, defined in Proposition \ref{prop: properties L}, on $\mathbf{n}$.
\par We start with $\lambda_L$ and we recall that the linear Boltzmann operator is an operator from $L^2_v(\boldsymbol\Maxwni)$ to $L^2_v(\boldsymbol\Maxwni)$ defined as $\mathbf{L}(\mathbf f) =(L_1(\mathbf f),\dots,L_N(\mathbf f))$ given by \eqref{L}--\eqref{Lij}. We shall follow the decomposition introduced in \cite{DauJunMouZam,BriDau}, which reads
\begin{equation}\label{decomposition L}
\mathbf{L} = \mathbf{L^m} + \mathbf{L^b}\quad\mbox{with}\quad \left\{\begin{array}{l}\disp{L^m_i = L_{ii}(f_i)} \\ \disp{L^b_i=\sum\limits_{j\neq i}L_{ij}(f_i,f_j).} \end{array}\right.
\end{equation}
Physically, $\mathbf{L^m}$ encodes all the inner interactions within a unique species whereas $\mathbf{L^b}$ takes care of all the bi-species interactions. Of important note is the fact that the basis of most of the works on the Boltzmann equation in perturbative settings require a stronger form of spectral gap for the linear operator. Namely, one needs a coercivity estimate with a gain of weight $\nu(v)$, the collision frequency. {If this is of core importance when solving the Boltzmann equation, it would however give a suboptimal negative property for the Fick matrix we are about to build up, and we therefore only derive a standard spectral gap for $\mathbf{L}$. We  thus mimick the proof of \cite{BriDau} to give a standard spectral gap result. This allows to derive a larger negative feedback, as one can see in the case of mono-species linear operator \cite{BarMou,Mou1}.

\textbf{The special case of mono-species operators.}
Since the operators $L_{ii}$ represent the interactions happening inside each species individually, they are mono-species Boltzmann linear operators.
 Such operators have an explicit spectral gap. In the case when linearizing around the normalized Maxwellian $\Maxwglob_0 = \frac{1}{(2\pi)^{d/2}}e^{-\frac{\abs{v}^2}{2}}$, then the spectral gap of the linear Boltzmann operator $L_0$ has been computed in \cite[Theorem 1.1]{BarMou}.

\begin{theorem} [Spectral gap for $L_0$]
\label{theo:mono normalized}
Let $B = \Phi(\abs{v-v_*})b(\cos \theta)$ be a collision kernel for particles of mass $1$ and satisfying
\begin{itemize}
\item[(i)]$\disp{\exists R,c_\Phi >0,\:\forall r\geq R, \quad \Phi(r)\geq c_\Phi}$;
\item[(ii)]$\disp{c^b := \inf_{\sigma_1,\sigma_2\in\S^{d-1}}\int_{\S^{d-1}}\min\big\{	b(\sigma_1\cdot\sigma_3),b(\sigma_2\cdot\sigma_3)\big\}\:d\sigma_3 > 0}$.
\end{itemize}
Then, for any $h$ in $L^2_v(\mu_0^{-1/2})$, the following holds
$$\langle h, L_0(h)\rangle_{L^2_v(\mu_0^{-1/2})} \leq -\lambda_0(c_\Phi,c_b,R) \norm{h-\pi_{L_0}(h)}^2_{L^2_v(\mu_0^{-1/2})}$$
where the spectral gap $\lambda_0$ is given by
\begin{equation}\label{lambda0}
\lambda_0(c_\Phi,c_b,R) = \frac{c_\Phi c_b e^{-4R^2}}{96}.
\end{equation}
\end{theorem}
\bigskip

We now would like to apply the theorem above in the case of a more general Maxwellian $\Maxwni_i$ which leads to a linear Boltzmann operator $L_{ii}$. This is the purpose of the following corollary.

\begin{cor}[Spectral gap for $L_{ii}$]
\label{cor:mono general}
For any $h$ in $L^2_v(\Maxwni_i^{-1/2})$, the following holds
$$\langle h, L_{ii}(h)\rangle_{L^2_v(\Maxwni_i^{-1/2})} \leq -\lambda_i \norm{h-\pi_{L_{ii}}(h)}^2_{L^2_v(\Maxwni_i^{-1/2})},$$
where the spectral gap depends on $\lambda_0$ defined by \eqref{lambda0} and is given by
\begin{equation}\label{lambdai}
\lambda_i = \frac{\lambda_0(c_{\Phi,i},c_{b,i},R_i)}{m_i^{\gamma/2}}.
\end{equation}
\end{cor}
\bigskip

\begin{proof}[Proof of Corollary \ref{cor:mono general}]
We come back to the explicit definition of $L_{ii}$ and write
\begin{multline*}
 \langle h, L_{ii}(h)\rangle_{L^2_v(\Maxwni_i^{-1/2})} = -\frac{1}{4}\int_{\R^{2d}\times\S^{d-1}}B_{ii} (|v-v_*|,\cos\theta) \Maxwni_i\Maxwni_i^*\\
 \cro{\pa{\frac{h}{\Maxwni_i}}^{'*}+\pa{\frac{h}{\Maxwni_i}}^{'}-\pa{\frac{h}{\Maxwni_i}}^{*}-\pa{\frac{h}{\Maxwni_i}}}dvdv_*d\sigma. 
\end{multline*}
Applying the change of variable $ (\sqrt{m_i}v,\sqrt{m_i}v_*)\mapsto (w,w_*)$ maps $\Maxwni_i(v)$ to $n_i m_i^{d/2}\Maxwglob_0(w)$ and $v'$ into $m_i^{-1/2}w'$ (same for $v'_*$ and $w'_*$), where $w'$ and $w'_*$ are the pre-collisional velocities giving $w$ and $w_*$ after a collision between particles of mass $1$. Moreover, $\Phi(\abs{v-v_*})$ becomes $m_i^{-\gamma/2}\Phi(\abs{w-w_*})$ thanks to hypothesis $(H3)$. Denoting $\tilde{h}(w) = h(w/\sqrt{m_i})$ yields
$$\langle h, L_{ii}(h)\rangle_{L^2_v(\Maxwni_i^{-1/2})} = \frac{n_i }{m_i^{\gamma/2}}\langle \tilde{h}, L_0(\tilde{h})\rangle_{L^2_v(\Maxwglob_0^{-1/2})}$$
which implies, after applying Theorem \ref{theo:mono normalized}
\begin{eqnarray*}
\langle h, L_{ii}(h)\rangle_{L^2_v(\Maxwni_i^{-1/2})} &\leq& - \frac{n_i }{m_i^{\gamma/2}}\lambda_0(c_{\Phi,i},c_{b,i},R_i)\int_{\R^d}h\pa{\frac{w}{\sqrt{m_i}}}^2\Maxwglob_0(w)dw
\\&\leq& -\frac{\lambda_0(c_{\Phi,i},c_{b,i},R_i)}{m_i^{\gamma/2}}\norm{h}^2_{L^2_v(\Maxwni_i)},
\end{eqnarray*}
where we made the change of variable $w \mapsto \sqrt{m_i}v$. This concludes the proof.
\end{proof}

\textbf{The general case of multi-species operators.}
Given concentrations $\mathbf{n}$ and masses $\mathbf{m}$, we associate the total concentration $c_\infty(\mathbf{n})$ and total density $\rho_\infty(\mathbf{n})$ defined by
\begin{equation}\label{def total concentration and density}
c_\infty(\mathbf{n}) = \sum\limits_{i=1}^Nn_i \quad\mbox{et}\quad \rho_\infty(\mathbf{m},\mathbf{n}) = \sum\limits_{i=1}^N m_in_i.
\end{equation}

\begin{prop}\label{prop:lambdaL explicit}
The spectral gap $\lambda_L$ of $\mathbf L$ is given by
\begin{equation}\label{lambdaL explicit}
\lambda_L = \frac{\Lambda(\mathbf{m},\mathbf{n}) \eta_0}{20N\max\br{\rho_\infty(\mathbf{m},\mathbf{n}),6c_\infty(\mathbf{n})}}
\end{equation}
where $c_\infty(\mathbf{n})$ and $\rho_\infty(\mathbf{m},\mathbf{n})$ are defined in \eqref{def total concentration and density}, $\lambda_i$ in \eqref{lambdai}, $\Lambda(\mathbf{m},\mathbf{n})$ in \eqref{Lambda} and $\eta_0$ in \eqref{eta0}. 
\end{prop}

\begin{proof}
 We recall the decomposition \eqref{decomposition L} for the linear operator $\mathbf{L}$. Thanks to Corollary \ref{cor:mono general} we can estimate the mono-species part in the following way
\begin{eqnarray}
\langle \mathbf{f}, \mathbf{L^m}(\mathbf{f})\rangle_{L^2_v(\boldsymbol\Maxwni^{-1/2})} &=& \sum\limits_{i=1}^N \langle f_i, L_{ii}(f_i)\rangle_{L^2_v(\Maxwni_i^{-1/2})} \nonumber
\\&\leq& -\sum\limits_{i=1}^N \lambda_i \norm{f_i-\pi_{L_{ii}}(f_i)}^2_{L^2_v(\Maxwni_i^{-1/2})}\nonumber
\\&\leq&-\min\limits_{1\leq i \leq N}\br{\lambda_i}\norm{\mathbf{f}-\pi_{\mathbf{L^m}}(\mathbf{f})}^2_{L^2_v(\boldsymbol\Maxwni^{-1/2})} .\label{spectral gap Lm}
\end{eqnarray}
Therefore, the constant $C_1$ in \cite[Lemma 3.4]{BriDau} is replaced by $\min_i\br{\lambda_i}$.
\par It remains to estimate the cross-interactions $\mathbf{L^b}$. We closely follow the computations of \cite[Proof of Theorem 3.3]{BriDau}. Note however that some constants are different, since we do not work in the more regular space $L^2_v(\boldsymbol\nu^{1/2}\boldsymbol\Maxwni^{-1/2})$ but remain in $L^2_v(\boldsymbol\Maxwni^{-1/2})$.
In \cite{BriDau}, equation (3.23) is modified by changing $4\eta$ into $4\eta\min_i\inf_v\br{\nu_i(v)}$ and keeping the norm considered here.
It also changes $k_0$ in (3.24) of \cite{BriDau} into $10N\max_{k,\ell}\abs{\langle \boldsymbol\Psi_k,\boldsymbol\Psi_\ell\rangle_{L^2_v(\boldsymbol\Maxwni^{-1/2})}}$, which is equal to $10N$: indeed, the scalar product is either $0$ or $1$ since $(\boldsymbol{\Psi_\ell})_{1\leq \ell \leq (d+2)N}$ is an orthonormal basis of $\Ker(\mathbf{L^m})$ in $L^2_v(\boldsymbol\Maxwni^{-1/2})$.
We therefore obtain
\begin{multline*}
\langle \mathbf{f}, \mathbf{L}(\mathbf{f})\rangle_{L^2_v(\boldsymbol\Maxwni^{-1/2})} \\
\leq- \frac{\Lambda(\mathbf{m},\mathbf{n})  \eta}{20N\max\br{\rho_\infty(\mathbf{m},\mathbf{n}),6c_\infty(\mathbf{n})}}\norm{\mathbf{f}-\pi_{\mathbf{L}}(\mathbf{f})}^2_{L^2_v(\boldsymbol\Maxwni^{-1/2})} 
\\
-\left(\min\limits_i\br{\lambda_i} -4\eta \min_i\inf_v\br{\nu_i(v)} -\frac{\Lambda(\mathbf{m},\mathbf{n})  \eta}{10N\max\br{\rho_\infty(\mathbf{m},\mathbf{n}),6c_\infty(\mathbf{n})}}\right)\norm{\mathbf{f}-\pi_{\mathbf{L^m}}(\mathbf{f})}^2_{L^2_v(\boldsymbol\Maxwni^{-1/2})},
\end{multline*}
where 
\begin{multline}\label{Lambda}
\Lambda(\mathbf{m},\mathbf{n}) = \frac{1}{4}\min_{1\le i,j\le n}\int_{\R^{2d}\times\S^{d-1}}m_i^2B_{ij}(\abs{v-v_*},\sigma)\\
\min\left\{\frac13|v-v'|^2,(|v'|^2-|v|^2)^2\right\}\Maxwni_i\Maxwni_j^* \:dvdv_*d\sigma,
\end{multline}
and $\eta\in (0,1]$ is chosen less or equal than $\eta_0$ defined by
\begin{equation}\label{eta0}
  \eta_0 =\min\br{1,\:\frac{10N\min\limits_i\br{\lambda_i}\max\br{\rho_\infty(\mathbf{m},\mathbf{n}),6c_\infty(\mathbf{n})}}{C_2+40N\max\br{\rho_\infty(\mathbf{m},\mathbf{n}),6c_\infty(\mathbf{n})}\min\limits_i\inf\limits_v\br{\nu_i(v)}}}
  \end{equation}
in order to guarantee the positivity of the parenthesis in front of $\norm{\mathbf{f}-\pi_{\mathbf{L^m}}(\mathbf{f})}^2_{L^2_v(\boldsymbol\Maxwni^{-1/2})}$.
This yields the claimed explicit value for the spectral gap $\lambda_L$ of $\mathbf L$.
\end{proof}

It remains to give an explicit estimate for $C_L$, defined in Proposition \ref{prop: properties L}.

\begin{prop}\label{prop:CL explicit}
There exists $C_0(\mathbf{m}) >0$ such that the boundedness constant $C_L$ of $\mathbf L$ is given by
\begin{equation}\label{CL explicit}
C_L = C_0(\mathbf{m})\max\limits_{1\leq i \leq n} \abs{n_i}.
\end{equation}
\end{prop}

\begin{proof}
The proof is rather straightforward using the decomposition \eqref{decomposition L} of $L_i$. Indeed, $\Maxwni_i = n_i\Maxwglob_i$ with $\Maxwglob_i$ independent of any $n_j$. As $L^2$-estimates on $Q_{ij}$ leading to boundedness properties of $\mathbf{L}$ have been obtained using direct triangular inequalities \cite{BriDau}, the result follows.
\end{proof}
\bigskip

\section{Formal convergence of the Boltzmann equation to the Fick one}\label{subsec:Fick matrix}

We will now derive formally the Fick equation as the hydrodynamical limit of the multispecies Boltzmann equation in the diffusive scaling \eqref{eq:multiBE}.
Let us write the following expansion, for $1\leq i \leq N$
\[F_i^\eps = \Maxwni_i+ \eps f_i^\eps = n_i \Maxwglob_i + \eps f_i^\eps,\]
where the distribution functions $F_i^\eps$ satisfy the Boltzmann equation with the diffusive scaling \eqref{eq:multiBE}.
We first obtain the mass conservation equation by integrating the equation on $\R^d$ with respect to $v$, and keeping the first order terms (at order $\eps^1$), it leads to
\begin{equation}\label{eq:massFick}
\partial_t n_i + \nabla_x \cdot J_i = 0,
\end{equation}
where the fluxes $J_i$ are defined by
\begin{equation}\label{eq:Jidef}
J_i = \int_{\R^d} f_i^\eps v\, dv.
\end{equation}
Further, recalling that $\mathbf{Q}(\mathbf{M},\mathbf{M}) =0$, we inject the above expansion in the Boltzmann equation and we keep the terms at order $\eps^0$ to write
\[\Maxwglob_i v\cdot \nabla_x n_i = \sum_{j=1}^N Q_{ij} (\Maxwni_i,f_j^\eps) + Q_{ij}(f_i^\eps,\Maxwni_j) = L_i (\mathbf f^\eps).\]

Denoting $\mathbf W = (W_i)_{1\leq i \leq N}$ the vector defined by $W_i = \Maxwglob_i v\cdot \nabla_x n_i$, this relation becomes $ L_i (\mathbf f^\eps) = W_i$ for $1\leq i \leq N$.
On the condition that $\pa{\Maxwglob_i v\cdot \nabla_x n_i}_{1\leq i \leq N}$ belongs to $\mbox{Ker}(\mathbf{L})^\bot$ in $L^2_v\pa{\mathbf{M}}$, this equation can be rewritten in the vectorial form $\mathbf f^\eps = {\mathbf L}^{-1} \mathbf W$. This condition means, by integrating against $\boldsymbol\phi_k$, defined in $\eqref{piL}$ for $N+1\leq k \leq N+d$, that for any $1\leq k \leq d$,
$$ 0 = \sum\limits_{i=1}^N\int_{\R^d} \mu_iv\cdot\nabla_xn_i v_km_idv = \partial_k \pa{\sum\limits_{i=1}^N m_in_i(t,x)},$$
which, in other terms, imposes
\begin{equation}\label{eq:gradient mn}
\forall (t,x) \in \R^+\times\T^d,\quad \nabla_x\langle\mathbf{m},\mathbf{n}(t,x)\rangle=0.
\end{equation}
Then with $\mathbf f^\eps = {\mathbf L}^{-1} \mathbf W$, the $k$-th component of the flux \eqref{eq:Jidef} of species $i$ can be expressed, for $1\leq k \leq d$ and $1\leq i \leq N$
\begin{align*}
 J_i^{(k)} &= \int_{\R^d} [{\mathbf L}^{-1} \mathbf W]_i v_k \,dv =  \int_{\R^d} [{\mathbf L}^{-1} \mathbf W]_i M_i v_k M_i^{-1} \,dv \\
 & =n_i\langle \mathbf C^{(i,k)}, {\mathbf L}^{-1} \mathbf W \rangle_{L^2_v(\mathbf{M}^{-1/2})},
\end{align*}
where we defined the tensor $ \mathbf C^{(i,k)} = (\mu_i v_k \delta_{ij})_{1\leq j \leq N}$.

The operator $\mathbf L^{-1}$ is self-adjoint on its domain $ (\Ker (\mathbf L))^\bot$. Since $ \mathbf C^{(i,k)} \notin (\Ker (\mathbf L))^\perp$, we have that
\begin{align*}
 J_i^{(k)} &=n_i\langle \mathbf C^{(i,k)} - \pi_{\mathbf L} (\mathbf C^{(i,k)}), {\mathbf L}^{-1} \mathbf W \rangle_{L^2_v(\mathbf{M}^{-1/2})}\\
 &= n_i\left\langle {\mathbf L}^{-1} \left(\mathbf C^{(i,k)} - \pi_{\mathbf L} (\mathbf C^{(i,k)}) \right), \mathbf W \right\rangle_{L^2_v(\mathbf{M}^{-1/2})}\\
 & =  \sum_{j=1}^N n_i \int_{\R^d} \left[{\mathbf L}^{-1} \left(\mathbf C^{(i,k)} - \pi_{\mathbf L} (\mathbf C^{(i,k)}) \right)\right]_j W_j M_j^{-1} \,dv.
\end{align*}
We can compute $W_j$ in the following way
\[W_j = \Maxwglob_j v \cdot \nabla_x n_j = \sum_{\ell=1}^d \Maxwglob_j v_\ell \partial_{x_\ell} n_j = \sum_{\ell=1}^d C^{(j,\ell)}_j \partial_{x_\ell} n_j .\]
Thus, the flux $ J_i^{(k)}$ becomes
\begin{align*}
  J_i^{(k)}  & =  \sum_{j=1}^N \sum_{\ell=1}^d n_i\int_{\R^d} \left[{\mathbf L}^{-1} \left(\mathbf C^{(i,k)} - \pi_{\mathbf L} (\mathbf C^{(i,k)}) \right)\right]_j C^{(j,\ell)}_j \partial_{x_\ell} n_j M_j^{-1} \,dv\\
&  = \sum_{j=1}^N \sum_{\ell=1}^d n_i\left\langle {\mathbf L}^{-1} \left(\mathbf C^{(i,k)} - \pi_{\mathbf L} (\mathbf C^{(i,k)}) \right) , \mathbf C^{(j,\ell)} \right\rangle_{L^2_v( M_j^{-1/2})} \partial_{x_\ell} n_j .
\end{align*}

\begin{lemma}\label{lem:cij}
 The quantities $\left\langle {\mathbf L}^{-1} \left(\mathbf C^{(i,k)} - \pi_{\mathbf L} (\mathbf C^{(i,k)}) \right) , \mathbf C^{(j,\ell)} \right\rangle_{L^2_v( \mathbf{M}^{-1/2})} $, defined for $1\leq k,\ell\leq d$ and $1\leq i,j\leq N$, satisfy the following properties:
\begin{enumerate}
 \item[(a)] $\left\langle {\mathbf L}^{-1} \left(\mathbf C^{(i,k)} - \pi_{\mathbf L} (\mathbf C^{(i,k)}) \right) , \mathbf C^{(j,\ell)} \right\rangle_{L^2_v( \mathbf{M}^{-1/2})} =0$ for any $\ell\neq k$;
 \item[(b)] $\left\langle {\mathbf L}^{-1} \left(\mathbf C^{(i,k)} - \pi_{\mathbf L} (\mathbf C^{(i,k)}) \right) , \mathbf C^{(j,k)} \right\rangle_{L^2_v( \mathbf{M}^{-1/2})} $ is independent of $k$, and thus only depends on $i$ and $j$. 
\end{enumerate}
\end{lemma}

\begin{proof}
The idea of the proof is to use the result proved in \cite[Prop. 1 and 2]{BGP1} for coefficients of the same form but involving ${\mathbf L}$ instead of ${\mathbf L}^{-1}$. 
To this end, let 
\[E = \mathop{\text{Vect}}\limits_{\stackrel{1\leq i\leq N}{{\tiny 1\leq k\leq d}}} \mathbf C^{(i,k)},\]
which is of course of finite dimension, and consider $\Lambda$ the restriction of $\mathbf L$ to $E$. We have the following decomposition of 
\[E = \mathop{\text{Vect}}\limits_{\tiny 1\leq k\leq d} \boldsymbol \phi_{N+d} \oplus F,\]
 where $F = \Image \Lambda\subset \Image L$.  Thus, we can define $\Lambda^{-1} : F \to F$, which of course coincides with $\mathbf L^{-1}$ on $F$. Now, observe that for any $1\leq k\leq d$ and $1\leq i\leq N$, $\mathbf C^{(i,k)} - \pi_{\mathbf L} (\mathbf C^{(i,k)}) \in F$. This means that there exist $\mathbf D^{(i,k)}$, for any $1\leq k\leq d$ and $1\leq i\leq N$, such that $\Lambda^{-1} (\mathbf C^{(i,k)} - \pi_{\mathbf L} (\mathbf C^{(i,k)})) = \mathbf D^{(i,k)}$.
This means that 
\[\left\langle {\mathbf L}^{-1} \left(\mathbf C^{(i,k)} - \pi_{\mathbf L} (\mathbf C^{(i,k)}) \right) , \mathbf C^{(j,\ell)} \right\rangle_{L^2_v( \mathbf{M}^{-1/2})}  = \left\langle \mathbf D^{(i,k)}  , \Lambda (\mathbf D^{(j,\ell)} )\right\rangle_{L^2_v( \mathbf{M}^{-1/2})}.\]
Since $\Lambda $ and $\mathbf L$ coincide on $F$, using \cite[Prop. 1 and 2]{BGP1}, we know that these coefficients are zero if $k\neq \ell$, and do not depend on $k=\ell$, which proves the claimed result.
\end{proof}

Finally, denoting, for any value of $1\leq k \leq d$,
\[a_{ij} = n_i\left\langle {\mathbf L}^{-1} \left(\mathbf C^{(i,k)} - \pi_{\mathbf L} (\mathbf C^{(i,k)}) \right) , \mathbf C^{(j,k)} \right\rangle_{L^2_v( \mathbf{M}^{-1/2})},\]
we have 
\[J_i^{(k)} =  \sum_{j=1}^N a_{ij} \partial_{x_k} n_j,\]
which means that the fluxes satisfy the vectorial relation
\begin{equation}\label{eq:fluxFick}
\mathbf J = \mathbf A \nabla_x \mathbf n, 
\end{equation}
where $\mathbf J = (J_i)_{1\leq i \leq N}$ and $\mathbf A = (a_{ij})_{1\leq i,j\leq N}$.
Recalling the definition of $\mathbf C^{(i,k)}$ and property $(b)$ in Lemma \ref{lem:cij}, we can rewrite
\begin{equation}\label{eq:FickMatrix}
 a_{ij} = n_i \langle \mathbf{L}^{-1}\pa{\bar{v}\Maxwglob_i\mathbf{e_i} - \pi_\mathbf{L}(\bar{v}\Maxwglob_i\mathbf{e_i})} \:, \: \bar{v}\Maxwglob_j\mathbf{e_j} - \pi_\mathbf{L}(\bar{v}\Maxwglob_j\mathbf{e_j}) \rangle_{L^2_v(\mathbf{M}^{-1/2})},
\end{equation}
with $\bar{v} = \frac{1}{d}(v_1+\dots + v_d)$, and thus the matrix $\mathbf{A}$ as
\begin{equation}\label{eq:Fick matrix}
\mathbf{A}(\mathbf{n}) = \mathbf{N}(\mathbf{n})\mathbf{\bar{A}}(\mathbf{n})
\end{equation}
where
$$N_{ij}(\mathbf{n}) = n_i \delta_{ij} \quad\mbox{and}\quad \bar{a}_{ij} = \Big\langle \mathbf{L}^{-1}\pa{\bar{v}\Maxwglob_i\mathbf{e_i} - \pi_\mathbf{L}(\bar{v}\Maxwglob_i\mathbf{e_i})} \:, \: \bar{v}\Maxwglob_j\mathbf{e_j} - \pi_\mathbf{L}(\bar{v}\Maxwglob_j\mathbf{e_j}) \Big\rangle_{L^2_v(\mathbf{M}^{-1/2})}.$$
Of course, thanks to the symmetry invariance, $\bar{v}$ could be replaced by any coordinate $v_i$.
Now, combining relation \eqref{eq:fluxFick} with the mass conservation equation \eqref{eq:massFick} leads to the Fick equation 
\begin{equation}\label{eq:Fick eq}
 \partial_t \mathbf n + \nabla_x\cdot \pa{\mathbf A(\mathbf n) \nabla_x \mathbf n} =0.
\end{equation}
{Let us emphasize that we have \textit{a priori} preservation of mass by integrating over $\T^d$
$$\forall i \in \br{1,\dots,N},\quad \frac{d}{dt}\int_{\T^d}n_i(t,x)dx = 0,$$
which in turns implies
$$0=\frac{d}{dt}\pa{\sum\limits_{i=1}^N m_i\int_{\T^d}n_i(t,x)dx} = \frac{d}{dt}\int_{\T^d}\pa{\sum\limits_{i=1}^N m_in_i(t,x)}dx.$$
The integrand being constant in $x$ from $\eqref{eq:gradient mn}$, we conclude that it is also constant in time. The limiting Fick equation must thus be supplemented with
\begin{equation}\label{eq:constant mn}
\exists C_{\rm Fick},\:\forall (t,x) \in \R^+\times\T^d,\quad \sum\limits_{i=1}^Nm_in_i(t,x) = C_{\rm Fick}.
\end{equation}
}

\section{Properties of the Fick matrix}\label{subsec:properties Fick matrix}

As for the linear Boltzmann operator, the Fick matrix is defined locally in time and space,
and for the sake of readability, we do not write down the dependences on $(t,x)$.
However, we track down the explicit dependences on $\mathbf{n} = \mathbf{n}(t,x)$ and $\mathbf{m}$.
We recall the Fick matrix associated to $N$ species is given by
$$\mathbf{A}(\mathbf{n)} = \mathbf{N}(\mathbf{n})\mathbf{\bar{A}}(\mathbf{n})$$
where $\mathbf{N}(\mathbf{n})$ is the diagonal matrix $\mbox{diag}(n_1,\dots,n_N)$ and $\mathbf{\bar{A}}(\mathbf{n})$ is defined by
\begin{equation}\label{eq: Fick matrix}
\mathbf{\bar{A}}= \pa{\Big\langle \mathbf{L}^{-1}\pa{\bar{v}\Maxwglob_i\mathbf{e_i} - \pi_\mathbf{L}(\bar{v}\Maxwglob_i\mathbf{e_i})} \:, \: \bar{v}\Maxwglob_j\mathbf{e_j} - \pi_\mathbf{L}(\bar{v}\Maxwglob_j\mathbf{e_j}) \Big\rangle_{L^2_v(\mathbf{M}^{-1/2})}}_{1\leq i,j \leq N}. 
\end{equation}
\par The goal of the present section is to understand if $\mathbf{\bar{A}}$ is coercive outside  its kernel. It comes from the properties of the linear operator $\mathbf{L}^{-1}$ {which exists when $n_i(t,x) \geq n_0 >0$. Moreover, since all $n_i$ are positive, $\mathbf{N}(\mathbf{n})$ is invertible.}

\begin{prop}\label{prop:properties Fick matrix}
The matrix $\mathbf{\bar{A}}$ is symmetric and $ \Ker\pa{\mathbf{\bar{A}}} = \emph{\mbox{Span}}\pa{\mathbf{n}\mathbf{m}}$.
\end{prop}

\begin{proof}[Proof of Proposition \ref{prop:properties Fick matrix}]
The symmetry property directly comes from the self-adjointness of $\mathbf{L}^{-1}$ in $L^2_{\mathbf{M}^{-1/2}}$, Proposition \ref{prop: spectral gap L-1}.
\par Let us now consider $\mathbf{X}$ in $\Ker\pa{\mathbf{\bar{A}}}$. For all $1\leq i \leq N$, we have
$$0 = \sum\limits_{j=1}^N \bar{A}_{ij}X_j = \Big\langle \mathbf{L}^{-1}\pa{\bar{v}\Maxwglob_i\mathbf{e_i} - \pi_\mathbf{L}(\bar{v}\Maxwglob_i\mathbf{e_i})} \:, \: \sum\limits_{j=1}^N X_j\pa{\bar{v}\Maxwglob_j\mathbf{e_j} - \pi_\mathbf{L}(\bar{v}\Maxwglob_j\mathbf{e_j})} \Big\rangle_{L^2_v(\mathbf{M}^{-1/2})}.$$
Summing over $i$ the previous relation multiplied by $X_i$  gives
$$\langle \mathbf{L}^{-1}\mathbf{Y}, \mathbf{Y}\rangle_{L^2_v\pa{\mathbf{M}^{-1/2}}}=0$$
where $\mathbf{Y} = \sum\limits_{j=1}^N X_j\pa{\bar{v}\Maxwglob_j\mathbf{e_j} - \pi_\mathbf{L}(\bar{v}\Maxwglob_j\mathbf{e_j})}$. The latter implies that $\mathbf{Y}$ belongs to $\Ker(\mathbf{L})$ which is fulfilled if and only if $\mathbf{X} $ belongs to $\mbox{Span}(\mathbf{n}\mathbf{m})$.
\end{proof}

We define $\boldsymbol\Pi_A$ the orthogonal projection in $\R^N$ on the kernel of $\A$ as well as its orthogonal $\boldsymbol\Pi^\bot_A$. For any $ \mathbf{X} \in \R^N$,
\begin{equation}\label{PiA}
 \boldsymbol\Pi_A(\mathbf{X}) = \langle \mathbf{n}\mathbf{m},\mathbf{X} \rangle \frac{\mathbf{n}\mathbf{m}}{\abs{\mathbf{n}\mathbf{m}}}
\quad \text{ and } \quad \boldsymbol\Pi^\bot_A(\mathbf{X}) = \mathbf{X}-\boldsymbol\Pi_A(\mathbf{X}).
\end{equation}
Since $\mathbf{\bar{A}}$ is symmetric, it has $N$ real eigenvalues. We shall prove that they are all negative and we give explicit bounds for these eigenvalues.

\begin{prop}\label{prop:eigenvalues Fick}
The matrix $\mathbf{\bar{A}}$ has $N-1$ nonzero eigenvalues denoted $\beta_1,\dots,\beta_{N-1}$. Moreover there  exists $C_1(\mathbf{m}) >0$ such that, for any $1\leq i \leq N-1$,
\begin{equation}\label{lambdaA}
  -\lambda_A(\mathbf{m},\mathbf{n}) \leq \beta_i <0 \quad \text{ where }\quad \lambda_A(\mathbf{m},\mathbf{n}) = \frac{C_1(\mathbf{m})}{\min\br{n_i}\lambda_L},
  \end{equation}
  and $\lambda_L$ defined by \eqref{lambdaL explicit} depends on $(\mathbf{m},\mathbf{n})$.
Moreover, there exists a function $\func{\beta_{\max}}{\R^{*+}}{\R^{*-}}$ such that if $\min\limits_{1\leq i \leq N}\br{n_i} > n_{\min}>0$ then, for any $1\leq i \leq N-1$,
$$ \beta_i <\beta_{\max}(n_{\min}) < 0. $$
\end{prop}

\begin{proof}[Proof of Proposition \ref{prop:eigenvalues Fick}]
First, the eigenvalues of $\mathbf{\bar{A}}$ are continuous functions of its coefficients, themselves being continuous functions of $\mathbf{m}$ and $\mathbf{n}$. Moreover, since the kernel of $\A$ is of dimension exactly $1$, the $\beta_i$ cannot change sign for varying $\mathbf{n}$ and $\mathbf{m}$.
Besides, we observe that in the case $m_i=m_j$ for all $1\leq i,j\leq N$,  implying $\Maxwglob_i=\Maxwglob_j$, these eigenvalues are equals. Indeed, by definition of $\mathbf{L}^{-1}$ (one can also see it as the core hypothesis in the Boltzmann model of indistinguishability of particles) we see that all the non-diagonal coefficients are equal as well as  all the diagonal ones. Therefore, we infer $\beta_i = \beta$ for all $1\leq i\leq N-1$ and compute, using Proposition \ref{prop: spectral gap L-1}
\begin{eqnarray*}
 (N-1)\beta  = \mbox{Tr}\pa{\mathbf{\bar{A}}} &=& N \Big\langle \mathbf{L}^{-1}\pa{\bar{v}\Maxwglob_1\mathbf{e_1} - \pi_\mathbf{L}(\bar{v}\Maxwglob_1\mathbf{e_1})} \:, \: \bar{v}\Maxwglob_1\mathbf{e_1} - \pi_\mathbf{L}(\bar{v}\Maxwglob_1\mathbf{e_1}) \Big\rangle_{L^2_v(\mathbf{M}^{-1/2})}
\\&\leq& -N\frac{\lambda_L}{C_L^2}\norm{\bar{v}\Maxwglob_1\mathbf{e_1} - \pi_\mathbf{L}(\bar{v}\Maxwglob_1\mathbf{e_1})}^2_{L^2_v(\mathbf{M}^{-1/2})}
\\&\leq& -N\frac{\lambda_L}{\max\br{n_i}C_L^2}\norm{\bar{v}\Maxwglob_1\mathbf{e_1} - \pi_\mathbf{L}(\bar{v}\Maxwglob_1\mathbf{e_1})}^2_{L^2_v(\boldsymbol\mu^{-1/2})}
\end{eqnarray*}
We thus deduce that the eigenvalues $\beta_1,\dots,\beta_{N-1}$ are all negative, for any values of $\mathbf{n}$ and $\mathbf{m}$. The existence of the continuous function $\beta_{\max}$ just comes from the continuity of the eigenvalues with respect to the $n_i$ when there are all strictly positive.

Using Cauchy-Schwarz inequality and Proposition \ref{prop: spectral gap L-1}  we can bound the eigenvalues from below
\begin{eqnarray*}
0 \geq \beta_i &\geq& \sum\limits_{j=1}^N \beta_j = \mbox{Tr}\pa{\mathbf{\bar{A}}}
\\&=& \sum\limits_{j=1}^N \Big\langle \mathbf{L}^{-1}\pa{\bar{v}\Maxwglob_j\mathbf{e_j} - \pi_\mathbf{L}(\bar{v}\Maxwglob_j\mathbf{e_j})} \:, \: \bar{v}\Maxwglob_j\mathbf{e_j} - \pi_\mathbf{L}(\bar{v}\Maxwglob_j\mathbf{e_j}) \Big\rangle_{L^2_v(\mathbf{M}^{-1/2})}
\\&\geq& -\frac{1}{\lambda_L}\sum\limits_{j=1}^N\norm{\bar{v}\Maxwglob_j\mathbf{e_j} - \pi_\mathbf{L}(\bar{v}\Maxwglob_j\mathbf{e_j})}^2_{L^2_v{\pa{\mathbf{M}^{-1/2}}}}
\geq - \frac{1}{\min\br{n_i}\lambda_L}\sum\limits_{j=1}^N d_j(\mathbf{m}),
\end{eqnarray*}
where we defined $d_j\pa{\mathbf m} = \norm{\bar{v}\Maxwglob_j\mathbf{e_j} - \pi_\mathbf{L}(\bar{v}\Maxwglob_j\mathbf{e_j})}^2_{L^2_v{\pa{\boldsymbol\mu^{-1/2}}}}$ which is independent of $\mathbf{n}$ and non-negative thanks to the definition of $\pi_L$. Denoting $C_1(\mathbf{m}) = \max\limits_{1\leq j \leq N} \br{d_j(\mathbf{m})}$ we obtain the desired lower bound.
\end{proof}

Further, we give some Sobolev estimates for the full matrix $\A$.

\begin{prop}\label{prop:sobolev A}
Let $s>d/2$ be an integer, let $\mathbf{n_\infty},\:\delta_m,\:\delta_M >0$ and $\mathbf{\tilde{n}}(t,x)$ in $H^s_x$ such that 
$$\delta_m \leq \mathbf{n}(t,x)= \mathbf{n_\infty} +\mathbf{\tilde{n}}(t,x)\leq \delta_M.$$
Then, for any multi-index $\abs{\ell}\leq s$, there exists a continuous function $P^\ell_s$ with $P^\ell_s(0)=0$ for $\abs{\ell} \geq 1$ and a constant $C(s,\delta_m,\delta_M) >0$ such that
\begin{equation}\label{eq:sobolev A}
\norm{\pl \mathbf{A}(\mathbf n)}_{L^2_x} \leq C(s,\delta_m,\delta_M)P^\ell_s(\norm{\mathbf{\tilde{n}}}_{H^s_x}),
\end{equation}
where $\partial_\ell$ denotes derivatives with respect to $x$.
\end{prop}

\begin{remark}
The proposition above is not very precise as we do not explicitly compute the function $P^\ell_s$, which would be a tedious calculation to make, since the exact expression of $\mathbf{L}^{-1}$ seems, at the very least, hard to explicit. It will however prove itself sufficient to construct a perturbative Cauchy theory for the Fick equation, the core feature being the fact that $P^\ell_s$ vanishes at $0$.
\end{remark}

\begin{proof}[Proof of Proposition \ref{prop:sobolev A}]
\par For $s>d/2$, the Sobolev embedding $H^s_x \subset L^\infty_x$ implies that the $H^s_x$-norm is an algebraic norm for product of functions.
Further, observe that the multispecies Boltzmann linear operator $\mathbf{L}$ around $\mathbf{n}(t,x)$ is a polynomial in terms of $\mathbf{n}(t,x)$, and thus, for $\mathbf{n}(t,x) >0$, its inverse and its derivatives are continuous in $\mathbf{n}$. As a consequence, $\mathbf{n}\mapsto \mathbf{A}(\mathbf{n})$ is infinitely many times differentiable on $\R^d\setminus\br{0}$. Noticing that
$$\partial_\ell\cro{\mathbf{A}\pa{\mathbf{n_\infty} +\mathbf{\tilde{n}}(t,x)}} = \pa{\partial_\ell \mathbf{\tilde{n}}(t,x)} \cdot \pa{\nabla_{\mathbf{n}}\mathbf{A}\pa{\mathbf{n_\infty} +\mathbf{\tilde{n}}(t,x)}},$$
we deduce \eqref{eq:sobolev A} by continuity of $\nabla_n\mathbf{A}$ on the annulus $\delta_m \leq \abs{\mathbf{n}}\leq \delta_M$.
\end{proof}
\bigskip

\section{Perturbative Cauchy theory for the Fick equation}\label{sec:resolution Fick}
We recall Fick equation defined by $\eqref{eq:Fick eq}$ supplemented with the closure relation \eqref{eq:constant mn}
\begin{equation}\label{eq:Fick eq new}
 \left\{\begin{array}{l}\disp{\partial_t \mathbf n + \nabla_x\cdot \pa{\mathbf{N}(\mathbf{n})\mathbf{\bar{A}}(\mathbf n) \nabla_x \mathbf n} =0} \\ \disp{\langle \mathbf{m},\mathbf{n}\rangle = C_{\rm Fick}}\end{array}\right..
\end{equation}

Outside its kernel, we proved in Section \ref{subsec:properties Fick matrix} that the Fick matrix is strictly negative as long as $\mathbf n >0$, thus endowing \eqref{eq:Fick eq new} with a standard degenerate nonlinear parabolic structure, {if it was not for the dilatation by $\mathbf{N(n)}$}. Besides, $\A(\mathbf n)$ is continuous in $\mathbf{n}$ due to the continuity  of $\mathbf{L}^{-1}$.
The negativity of $\A(\mathbf n)$ is continuously controlled by $\min\br{n_i}$ as shown in Proposition \ref{prop:eigenvalues Fick}. The issue to obtain a complete Cauchy theory reduces to preventing the appearance of a singularity, \ie one of the $n_i(t,x)$ vanishing for some $(t,x)$. However, we are interested only in a perturbative regime around a global equilibrium $\mathbf{n_\infty} >0$, which means solutions of the form
$$\mathbf{n}(t,x) = \mathbf{n_\infty} + \mathbf{\tilde{n}}(t,x),$$
where $\tilde{\mathbf{n}}$ stands for a small perturbation. In this framework, if one controls the $L^\infty$ norm of  $\mathbf{\tilde{n}}$ globally in time by a control of the form
$$\exists C >0,\quad \norm{\mathbf{\tilde{n}}}_{L^\infty_{t,x}} \leq C \norm{\mathbf{\tilde{n}}(0,\cdot)}_{L^\infty_x},$$
then, for sufficiently small initial perturbation $\mathbf{\tilde{n}}(0,\cdot)$, one has 
$$\forall t \geq 0,\forall x \in \T^d, \quad \mathbf{n}(t,x) \geq \frac{1}{2}\mathbf{n_\infty} >0,$$
and {the Fick operator $\nabla_x\cdot (\mathbf{\bar{A}}(\mathbf n) \nabla_x \mathbf \cdot)$ acts like a small perturbation of the uniformly elliptic operator $\nabla_x\cdot(\mathbf{\bar{A}}(\mathbf{n_\infty})\nabla_x \cdot)$} outside its kernel with a lower bound $\beta_{\max}\pa{\frac{1}{2}\mathbf{n_\infty}}>0$ (Proposition \ref{prop:eigenvalues Fick}). As we shall see, the kernel part of a solution is entirely determined by its value at initial time thus allowing to fully estimate the solution \textit{a priori}.

{The perturbed equation reads
\begin{equation}\label{eq:perturbed Fick}
\left\{\begin{array}{l}\disp{\partial_t \mathbf{\tilde{n}} +\nabla_x\cdot\pa{\mathbf{N_\infty}\mathbf{\bar{A}}(\mathbf{n_\infty}+\mathbf{\tilde{n}})\nabla_x\mathbf{\tilde{n}}} = - \nabla_x\cdot\pa{\mathbf{\tilde{N}}\mathbf{\bar{A}}(\mathbf{n_\infty}+\mathbf{\tilde{n}})\nabla_x\mathbf{\tilde{n}}}} \\ \disp{\langle \mathbf{m},\mathbf{\tilde{n}}\rangle =0}\end{array}\right.
\end{equation}
where we straightforwardly denoted $\mathbf{N_\infty} = \mbox{diag}(n_{\infty 1},\dots,n_{\infty N})$ and $\mathbf{\tilde{N}} = \mbox{diag}(\tilde{n}_1,\dots,\tilde{n}_N)$.
We prove the following \textit{a priori} estimate.}

\begin{prop}\label{prop:a priori Fick}
Let $s>d/2$ be an integer, let $\delta>0$ and $\mathbf{n_\infty} >0$. 
There exist $\delta_s >0$ and $\lambda_s >0$ such that for any $\mathbf{\tilde{n}^\ini}$ in $H^s_x$ satisfying:
\begin{enumerate}[label=(\roman*)]
\item $\disp{\forall x \in\T^d,\:\mathbf{n_\infty} + \mathbf{\tilde{n}^\ini}(x) \geq \delta}$ and $\disp{\int_{\T^d}\mathbf{\tilde{n}^\ini}(x)dx=0}$;
\item {$\disp{\forall x \in \T^d,\:\sum\limits_{i=1}^N m_i\tilde{n}^\ini_i(x) = 0}$;}
\item {$\disp{\norm{\mathbf{\tilde{n}^\ini}}_{H^s_x} \leq \delta_s}$;}
\end{enumerate}
if $\mathbf{n}(t,x) = \mathbf{n_\infty} + \mathbf{\tilde{n}}(t,x)$ is a solution on $[0,T_{\max})$ to the Fick equation $\eqref{eq:Fick eq new}$ with initial datum $\mathbf{n_\infty}+\mathbf{\tilde{n}^\ini}(x)$, then for any $t \in [0,T_{\max})$, the following holds
\begin{enumerate}[label=(\alph*)]
\item $\disp{\forall x \in \T^d,\:\mathbf{n_\infty} + \mathbf{\tilde{n}}(t,x) \geq \delta }$ and $\disp{\int_{\T^d}\mathbf{\tilde{n}}(t,x)dx=0}$;
\item{$\disp{\norm{\mathbf{\tilde{n}}(t)}_{H^s_x}\leq \norm{\mathbf{\tilde{n}^\ini}}_{H^s_x}e^{-\lambda_s t}}$.}
\end{enumerate}
The constants $\delta_s$ and $\lambda_s$ only depend on $s$ and $\delta$.
\end{prop}

\begin{proof}[Proof of Proposition \ref{prop:a priori Fick}]
The fact that $\mathbf{\tilde{n}}$ has zero mean directly comes from the gradient form of the Fick equation.
Furthermore, since we have the continuous Sobolev embedding $H^s_x \subset L^\infty_x$, the positivity follows directly from the control in $H^s_x$, as long as $\delta_s$ is sufficiently small.
We thus solely have to establish $(c)$.
{\begin{remark}\label{rem:dilated}
Due to the presence of $\mathbf{N_\infty}$ it seems natural to work in the equivalent norm $L^2_x\pa{\mathbf{N_\infty}^{-\frac{1}{2}}}$. However, even dropping the nonlinear terms, a direct estimate yields
$$\frac{1}{2}\frac{d}{dt}\norm{\mathbf{\tilde{n}}}^2_{L^2_x\pa{\mathbf{N_\infty}^{-\frac{1}{2}}}} = \langle \mathbf{\bar{A}}\nabla_x\mathbf{\tilde{n}}, \nabla_x\mathbf{\tilde{n}} \rangle_{L^2_x}\leq -\beta_{\rm max}(\delta)\norm{\boldsymbol\Pi_{\mathbf{A}}^\bot\pa{\nabla_x \mathbf{\tilde{n}}}}_{L^2_x}.$$
We do obtain a negative feedback but the kernel quantity 
$$\boldsymbol\Pi_{\mathbf{A}}\pa{\nabla_x \mathbf{\tilde{n}}} = \langle \mathbf{n}\mathbf{m},\nabla_x\mathbf{\tilde{n}}\rangle\frac{\mathbf{n}\mathbf{m}}{\abs{\mathbf{n}\mathbf{m}}}$$
cannot be easily controlled because of the dilatation. Indeed, it is not constant and it interacts with the orthogonal part, even at main order.  Therefore, we cannot use standard methods for degenerate parabolic equations.
\end{remark}}

\textbf{Rescalings in time and space.} The idea is thus to get rid of $\mathbf{N_\infty}$ by other means than working with a weighted norm. We shall see that a rescaling in time and space transforms $\eqref{eq:perturbed Fick}$ into a degenerate parabolic equation for which the projection onto the kernel remains constant in time. Let us define for $\alpha,\beta\in\R$ the function $\mathbf g = (g_i)_{1\leq i \leq N}$ by
$$ g_i(t,x) = \tilde{n}_i\pa{n_{\infty,i}^{\alpha}t,n_{\infty,i}^{\beta}x}, \qquad \forall (t,x) \in \R^+\times\pa{\mathbf{n_{\infty}^{-\beta}}\T^d}.$$
The function $\mathbf{g}$ satisfies
\begin{equation}\label{eq:rescaled eq}
\left\{\begin{array}{l} \disp{\partial_tg_i(t,x) + \nabla_x\cdot\cro{\sum\limits_{j=1}^N\frac{n_{\infty,i}^{1+\alpha}}{n_{\infty,j}^{2\beta}}\bar{a}_{ij}\pa{\mathbf{n_\infty}+\mathbf{g}}\nabla_x g_j} = -\nabla_x\cdot\cro{\sum\limits_{j=1}^N\frac{n_{\infty,i}^{\alpha}}{n_{\infty,j}^{2\beta}}g_i\bar{a}_{ij}\pa{\mathbf{n_\infty}+\mathbf{g}}\nabla_x g_j}}\\ \disp{\langle \mathbf{m},\mathbf{g} \rangle =0.}\end{array}\right.
\end{equation}
Choosing $1+\alpha = -2\beta$ now yields a symmetric matrix and thus a parabolic equation for $\mathbf{g}$. The new matrix $\pa{\frac{n_{\infty,i}^{1+\alpha}}{n_{\infty,j}^{2\beta}}\bar{a}_{ij}}_{1\leq i,j \leq N}$ is still degenerate. We have the following $L^2$-estimate, using the spectral gap of $\mathbf{\bar{A}}$ (Prop. \ref{prop:eigenvalues Fick})
\begin{equation}\label{eq:estimate L2} 
\begin{split}
\frac{1}{2}\frac{d}{dt}\norm{\mathbf{g}}^2_{L^2_x} &= \langle \mathbf{\bar{A}}\nabla_x\pa{\mathbf{n_\infty^{-2\beta}}\mathbf{g}}, \nabla_x\pa{\mathbf{n_\infty^{-2\beta}}\mathbf{g}} \rangle_{L^2_x} + \langle \mathbf{\bar{A}}\nabla_x\pa{\mathbf{n_\infty^{-2\beta}}\mathbf{g}},\mathbf{g} \nabla_x\pa{\mathbf{n_\infty^{\alpha}}\mathbf{g}} \rangle_{L^2_x}\\
&\leq -\beta_{\rm max}(\delta)\norm{\boldsymbol\Pi^\bot_{\mathbf{A}}\pa{\nabla_x\mathbf{n_\infty^{-2\beta}}\mathbf{g}}}^2_{L^2_x} + \langle \mathbf{\bar{A}}\nabla_x\pa{\mathbf{n_\infty^{-2\beta}}\mathbf{g}},\mathbf{g} \nabla_x\pa{\mathbf{n_\infty^{\alpha}}\mathbf{g}} \rangle_{L^2_x}.
\end{split}
\end{equation}

The freedom we gained compared to Remark \ref{rem:dilated} is the power $-2\beta$. Indeed, we now have
$$\boldsymbol\Pi_{\mathbf{A}}\pa{\nabla_x\mathbf{n_\infty^{-2\beta}}\mathbf{g}} = \langle \pa{\mathbf{n_\infty}+\mathbf{g}}\mathbf{m},\nabla_x\pa{\mathbf{n_\infty^{-2\beta}}\mathbf{g}}\rangle\frac{\pa{\mathbf{n_\infty}+\mathbf{g}}\mathbf{m}}{\abs{\pa{\mathbf{n_\infty}+\mathbf{g}}\mathbf{m}}},$$
and so, fixing $2\beta = 1$, it remains
\begin{eqnarray*}
\boldsymbol\Pi_{\mathbf{A}}\pa{\nabla_x\mathbf{n_\infty^{-2\beta}}\mathbf{g}} &=& \nabla_x\pa{\langle \mathbf{m},\mathbf{g}\rangle}\frac{\pa{\mathbf{n_\infty}+\mathbf{g}}\mathbf{m}}{\pa{\mathbf{n_\infty}+\mathbf{g}}\mathbf{m}}+ \langle \mathbf{m}\mathbf g,\nabla_x\pa{\mathbf{n_\infty^{-1}}\mathbf{g}}\rangle\frac{\pa{\mathbf{n_\infty}+\mathbf{g}}\mathbf{m}}{\abs{\pa{\mathbf{n_\infty}+\mathbf{g}}\mathbf{m}}}
\\&=& \langle \mathbf{m}\mathbf g,\nabla_x\pa{\mathbf{n_\infty^{-1}}\mathbf{g}}\rangle\frac{\pa{\mathbf{n_\infty}+\mathbf{g}}\mathbf{m}}{\abs{\pa{\mathbf{n_\infty}+\mathbf{g}}\mathbf{m}}},
\end{eqnarray*}
because of the second relation in \eqref{eq:rescaled eq}.
This implies that the projection $\boldsymbol\Pi_{\mathbf{A}}\pa{\nabla_x\mathbf{n_\infty^{-2\beta}}\mathbf{g}}$ is now at lower order for small $\mathbf{g}$
\begin{equation}\label{eq:control PiA L2}
\norm{\boldsymbol\Pi_{\mathbf{A}}\nabla_x\pa{\frac{\mathbf{g}}{\mathbf{n_\infty}}}}_{L^2_x} \leq \frac{\max \br{m_i}}{\min \br{n_{\infty,i}}}\norm{\mathbf{g}}_{L^2_x}\norm{\nabla_x\mathbf{g}}_{L^2_x}.
\end{equation}
We shall now prove an exponential decay for $g$ which will imply (c).
\par Let us consider $s > d/2$. We shall denote by $C$ any positive constant independent of $\mathbf{g}$. We use \eqref{eq:estimate L2}, $\eqref{eq:control PiA L2}$ and Cauchy-Schwarz inequality to obtain an $L^2_x$ estimate on $\mathbf g$ as follows
\begin{eqnarray*}
\frac{1}{2}\frac{d}{dt}\norm{\mathbf{g}}^2_{L^2_x} \leq - C\beta_{\max}(\delta)\cro{1-C\norm{\mathbf{g}}^2_{L^2_x}}\norm{\nabla_x\mathbf{g}}^2_{L^2_x} + CP^0_s(\norm{\mathbf{g}}_{H^s_x})\norm{\mathbf{g}}_{L^\infty_x}\norm{\nabla_x \mathbf{g}}^2_{L^2_x}
\end{eqnarray*}
where we used Proposition \ref{prop:sobolev A} to control $\mathbf{\bar{A}}$. The Sobolev embedding $H^s_x \subset L^\infty_x$ concludes
\begin{equation}\label{eq:estimate L2 end}
\frac{1}{2}\frac{d}{dt}\norm{\mathbf{g}}^2_{L^2_x} \leq - C\beta_{\max}(\delta)\cro{1- C\pa{\norm{\mathbf{g}}_{H^s_x}+P^0_s(\norm{\mathbf{g}}_{H^s_x})}\norm{\mathbf{g}}_{H^s_x}}\norm{\nabla_x\mathbf{g}}^2_{L^2_x}.
\end{equation}

Let $\ell$ be a multi-index such that $\abs{\ell}\leq s$ and let us take the $\partial_\ell$-derivative of $\eqref{eq:rescaled eq}$ and integrate against $\partial_\ell\mathbf{g}$. It yields
\begin{eqnarray*}
\frac{1}{2}\frac{d}{dt}\norm{\partial_\ell\mathbf{g}}^2_{L^2_x} &=& \langle \mathbf{\bar{A}}\nabla_x\pa{\frac{\partial_\ell\mathbf{g}}{\mathbf{n_\infty}}},\nabla_x\pa{\frac{\partial_\ell\mathbf{g}}{\mathbf{n_\infty}}}\rangle_{L^2_x}+\sum\limits_{\underset{\abs{\ell_1}\geq 1}{\ell_1+\ell_2=\ell}}\langle \partial_{\ell_1}\mathbf{\bar{A}}\nabla_x\pa{\frac{\partial_{\ell_2}\mathbf{g}}{\mathbf{n_\infty}}},\nabla_x\pa{\frac{\partial_\ell\mathbf{g}}{\mathbf{n_\infty}}}\rangle_{L^2_x}
\\&\quad&+\sum\limits_{\ell_1+\ell_2+\ell_3 = \ell}\langle \partial_{\ell_1}\mathbf{\bar{A}}\nabla_x\pa{\frac{\partial_{\ell_2} \mathbf{g}}{\mathbf{n_\infty}}},\pa{\frac{\partial_{\ell_3}\mathbf{g}}{\mathbf{n_\infty}}}\nabla_x\pa{\frac{\partial_{\ell}\mathbf{g}}{\mathbf{n_\infty}}}\rangle_{L^2_x}.
\end{eqnarray*}
Since $\langle \mathbf{m},\partial_\ell\mathbf{g} \rangle =0$ we can copy the arguments of the $L^2_x$-estimate for the first term on the right-hand side. The last two terms are estimated using Cauchy-Schwarz inequality, the Sobolev embedding $H^s_x\subset L^\infty_x$ {(which implies that $H^s_x$ is an algebraic norm)} and the Sobolev controls on $\mathbf{\bar{A}}$ from Proposition \ref{prop:sobolev A} and lead to the following estimate
\begin{equation*}
\frac{1}{2}\frac{d}{dt}\norm{\partial_\ell \mathbf{g}}^2_{L^2_x} \leq -\beta_{\rm max}(\delta)\cro{1-CP_s(\norm{\mathbf{g}}_{H^s_x})}\norm{\nabla_x\partial_\ell\mathbf{g}}^2_{L^2_x} + P_s(\norm{\mathbf{g}}_{H^s_x})\norm{\nabla_x\mathbf{g}}^2_{H^{s}_x},
\end{equation*}
where $P_s$ is a continuous function satisfying $P_s(0)=0$. Therefore, summing over $\abs{\ell}\leq s$, we get
\begin{equation}\label{eq:estimate Hs}
\frac{1}{2}\frac{d}{dt}\norm{\mathbf{g}}^2_{H^s_x} \leq -\beta_{\rm max}(\delta)\cro{1-CP_s(\norm{\mathbf{g}}_{H^s_x})}\norm{\nabla_x\mathbf{g}}^2_{H^s_x}.
\end{equation}
To conclude, since $P_s(0)=0$, there exists a ball $B(0,\eta)$ centered at $0$ and of radius $\eta>0$ such that for any $x \in B(0,\eta)$, $CP_s(x) \leq \frac12$. Thus, choosing $\mathbf g^\ini$ such that $\norm{\mathbf{g^\ini}}_{H^s_x}\in B(0,\eta)$, we ensure, using \eqref{eq:estimate Hs}, that $CP_s(\norm{\mathbf{g}}_{H^s_x}) \leq \frac12$ at all times. This implies that
$$\forall t \geq 0,\quad\frac{1}{2}\frac{d}{dt}\norm{\mathbf{g}}^2_{H^s_x} \leq -\frac{\beta_{\rm max}(\delta)}{2}\norm{\nabla_x\mathbf{g}}^2_{H^s_x}.$$
It remains to use assumption $(i)$ which states that $\mathbf{g}$ has a zero integral over the torus and we can thus apply Poincar\'e inequality:
$$\forall 0\leq \abs{\ell} \leq s,\quad \norm{\partial_\ell \mathbf{g}}_{L^2_x} \leq C_p\norm{\nabla_x\partial_\ell \mathbf{g}}_{L^2_x},$$
which yields
$$\forall t \geq 0,\quad\frac{1}{2}\frac{d}{dt}\norm{\mathbf{g}}^2_{H^s_x} \leq -C_p\frac{\beta_{\rm max}(\delta)}{2}\norm{\mathbf{g}}^2_{H^s_x}.$$
This concludes the proof thanks to Gr\"onwall's lemma.

\end{proof}

\section{Rigorous convergence towards the Fick equation}\label{sec:BE}

This section is devoted to the proof of the stability of the Fick Maxwellian
$$\mathbf{M}^{\boldsymbol\eps}(t,x,v) = \pa{\mathbf{n_\infty}+\eps \mathbf{\tilde{n}}(t,x)}\boldsymbol\mu(v)$$
for the multispecies Boltzmann equation.

\begin{proof}[Proof of Theorem \ref{theo:BE}]
The theorem is a direct application of a recent theorem \cite[Th. 2.4]{BonBri}, which we state below for the sake of readibility. In the following statement we denote
$$\mathbf{S^\eps} = \frac{1}{\eps}\partial_t\mathbf{M^\eps}+\frac{1}{\eps^2}v\cdot\nabla_x\mathbf{M^\eps}- \frac{1}{\eps^3}\mathbf{Q}(\mathbf{M^\eps},\mathbf{M^\eps})$$
the source term coming from a local linearization in $\eqref{eq:multiBE}$.

\begin{theorem}[Th. 2.4 of \cite{BonBri}]
\label{theo:BonBri}
Under the assumptions $(H1)-(H2)-(H3)-(H4)$ on the collision kernel, there exists an integer $s_0$, some constants $\delta_{\mbox{\footnotesize{fluid}}},\:\delta_B,\:C_B>0$, $\eps \in (0,1]$ and a norm
$$\norm{\cdot}^2_{\mathcal{H}^s_\eps} \sim \cro{\sum\limits_{0\leq \abs{\ell}\leq s}\norm{\partial^\ell_x\cdot}^2_{L^2_{x,v}\pa{\boldsymbol\mu^{-\frac{1}{2}}}} + \eps^2\sum\limits_{\underset{\abs{j}\geq 1}{0\leq \abs{\ell} + \abs{j} \leq s}}\norm{\partial^\ell_x\partial^j_{\important{v}} \cdot}^2_{L^2_{x,v}\pa{\boldsymbol\mu^{-\frac{1}{2}}}}}$$
such that, if we consider functions
\begin{itemize}
\item[(i)] $\disp{\mathbf{c}(t,x)=\mathbf{\bar{c}}+ \eps\mathbf{\tilde{c}}(t,x)}$ in $H^s_x$ with $\disp{\norm{\mathbf{\tilde{c}}}_{L^\infty_tH^s_x} \leq \delta_{\mbox{\footnotesize{fluid}}}}$;
\item[(ii)] $\disp{\mathbf{u}(t,x)=\mathbf{\bar{u}}(t,x)+ \eps\mathbf{\tilde{u}}(t,x)}$ in $H^{s-1}_x$ with $\disp{\nabla_x\cdot \mathbf{\bar{u}}=0 }$ and $\disp{\norm{\mathbf{\tilde{u}}}_{L^\infty_tH^{s-1}_x} \leq \delta_{\mbox{\footnotesize{fluid}}}}$;
\item[(iii)] a fluid Maxwellian $\disp{M_i^\eps(t,x) = c_i(t,x)\pa{\frac{m_i}{2\pi}}^{\frac{d}{2}}e^{-\frac{\abs{v - \eps u_i(t,x)}}{2}}}$ such that 
$$\norm{\boldsymbol\pi^\bot_{\mathbf{L}}\pa{\mathbf{S^\eps}}}_{\mathcal{H}^s_\eps}  = \mathcal{O}\pa{\frac{\delta_{\mbox{\footnotesize{fluid}}}}{\eps}} \quad\mbox{and}\quad \norm{\boldsymbol\pi_{\mathbf{L}}\pa{\mathbf{S^\eps}}}_{\mathcal{H}^s_\eps}  = \mathcal{O}\pa{\delta_{\mbox{\footnotesize{fluid}}}};$$
\item[(iv)] $\disp{\mathbf{f^{(\mbox{\footnotesize{in}})}}}$ in $\mathcal{H}^s_\eps$ with $\norm{\mathbf{f^{(\mbox{\footnotesize{in}})}}}_{\mathcal{H}^s_\eps} \leq \delta_B$ and $\disp{\norm{\int_{\T^d}\boldsymbol\pi_{\mathbf{L}_{\boldsymbol\mu}}(\mathbf{f^{(\mbox{\footnotesize{in}})}})dx}_{L^2_{x,v}\pa{\boldsymbol\mu^{-1/2}}} = \mathcal{O}\pa{\delta_{\mbox{\footnotesize{fluid}}}}}$ where $\boldsymbol\pi_{\mathbf{L}_{\boldsymbol\mu}}$ is the orthogonal projection in $L^2_v(\boldsymbol\mu^{-\frac{1}{2}})$ onto the kernel of $\mathbf{L}_{\boldsymbol\mu}$ the Boltzmann operator linearized around the global equilibrium state $\boldsymbol\mu$;
\end{itemize}
Then the multispecies Boltzmann equation $\eqref{eq:multiBE}$ with initial datum $\mathbf{F^{(\mbox{\footnotesize in})}} = \mathbf{M^\eps}(0,x) + \eps \mathbf{f^{(\mbox{\footnotesize{in}})}}(x,v) \geq 0$ 
possesses a unique perturbative solution $\mathbf{F^\eps}(t,x,v) = \mathbf{M}(t,x) + \eps \mathbf{f^\eps}(t,x,v) \geq 0$ with $\mathbf{f^\eps}$ belonging to $C^0\pa{\R_+;H^s_{x,v}(\boldsymbol\mu^{-\frac{1}{2}})}$ and it satisfies the stability property
$$\forall t \geq 0,\quad \norm{\mathbf{F^\eps}-\mathbf{M^\eps}}_{\mathcal{H}^s_\eps}(t) \leq \eps C_B. $$
All the constant are explicit and independent of $\eps$.
\end{theorem}

In the framework we consider here, assumption $(i)$ is satisfied taking $\mathbf{\bar{c}}=\mathbf{n_\infty}$, $\mathbf{\tilde{c}}= \mathbf{\tilde{n}}$ and $\mathbf{\bar{u}}=\mathbf{\tilde{u}}=0$. Moreover we directly see that since, for $\mathbf{u}=0$, the state $\mathbf{M^\eps}$ cancels the Boltzmann operator $\mathbf{Q}$,
$$\mathbf{S^\eps} = \frac{1}{\eps}\partial_t\mathbf{M^\eps}+\frac{1}{\eps^2}v\cdot\nabla_x\mathbf{M^\eps}- \frac{1}{\eps^3}\mathbf{Q}(\mathbf{M^\eps},\mathbf{M^\eps}) = \frac{1}{\eps}\partial_t\mathbf{M^\eps}+\frac{1}{\eps^2}v\cdot\nabla_x\mathbf{M^\eps}.$$
Moreover, we also have
$$\boldsymbol\pi_{\mathbf{L}}\pa{\mathbf{S^\eps}} = 0.$$
It thus remains to prove the following estimate
\begin{equation}\label{eq:final estimate}
\norm{\partial_t\mathbf{M^\eps}+\frac{1}{\eps}v\cdot\nabla_x\mathbf{M^\eps}}_{\mathcal{H}^s_\eps} \leq \delta_{\rm fluid}.
\end{equation}
The definition of the $\mathcal{H}^s_\eps$-norm and the choice of $\mathbf{M^\eps}$ imply that, if there exists a constant $C_{\rm fluid}>0$ such that
\begin{equation}\label{eq:control Meps}
\eps\norm{\partial_t \mathbf{\tilde{n}}}_{H^s_x} + \norm{\nabla_x\mathbf{\tilde{n}}}_{H^s_x} \leq C_{\rm fluid}\delta_{\rm fluid},
\end{equation}
then the estimate $\eqref{eq:final estimate}$ is satisfied.

From Theorem \ref{theo:Fick} with 
$\norm{\mathbf{\tilde{n}^{\rm (in)}}}_{H^{s+1}_x} \leq  C_{\rm fluid}\delta_{\rm fluid}/2$, we have that
\begin{equation}\label{control gradx}
\forall t \geq 0,\quad \norm{\nabla_x\mathbf{\tilde{n}}}_{H^{s}_x} \leq  \norm{\mathbf{\tilde{n}}}_{H^{s+1}_x} \leq \norm{\mathbf{\tilde{n}^{\rm (in)}}}_{H^{s+1}_x}e^{-\lambda_{s+1}t} \leq \frac{C_{\rm fluid}\delta_{\rm fluid}}{2}.
\end{equation}
Moreover, in order to control $\eps\norm{\partial_t\mathbf{\tilde{n}}}_{H^s_x}$, let us denote $C_A$ the constant (obtained from Proposition \ref{prop:sobolev A}) such that
\[ \norm{ \mathbf{A}\pa{\mathbf{n_\infty}+\eps\mathbf{\tilde{n}}} }_{H^{s+2}_x} \leq  C_A \norm{\mathbf{\tilde{n}}}_{H^{s+2}_x}.\]
If we take $\norm{\mathbf{\tilde{n}^{\rm (in)}}}_{H^{s+2}_x} \leq \sqrt{C_{\rm fluid}\delta_{\rm fluid}/(2C_A)}$, it leads to
$$\norm{\partial_t\mathbf{\tilde{n}}}_{H^s_x} = \norm{\nabla_x\cdot \cro{\mathbf{A}\pa{\mathbf{n_\infty}+\eps\mathbf{\tilde{n}}}\nabla_x\mathbf{\tilde{n}}}}_{H^s_x} \leq C_A \norm{\mathbf{\tilde{n}}}_{H^{s+2}_x}^2 \leq C_A \norm{\mathbf{\tilde{n}^{\rm (in)}}}_{H^{s+2}_x}^2 \leq \frac{C_{\rm fluid}\delta_{\rm fluid}}{2}.$$
Therefore,  $\eqref{eq:control Meps}$ is satisfied, and this concludes the proof of Theorem \ref{theo:BE}.
\end{proof}

\begin{remark}\label{rem:Hs+1}
We conclude this proof by indicating that the general result Theorem \ref{theo:BonBri} could in fact be rewritten under a weaker form with local-in-time $\mathcal{H}^s_\eps$-norms replaced by $L^2_{[0,T_{\max})}\mathcal{H}^s_\eps$. We refer explicitely to \cite[Equation (3.36)]{BonBri} that one could integrate in time. In that framework we would solely have to prove the following control
$$\eps\norm{\partial_t \mathbf{\tilde{n}}}_{L^2_tH^s_x} + \norm{\nabla_x\mathbf{\tilde{n}}}_{L^2_tH^s_x} \leq C_{\rm fluid}\delta_{\rm fluid}$$
where the second term is already dealt with using $\eqref{control gradx}$. We saw in Section \ref{sec:resolution Fick} that $\mathbf{\tilde{n}}$ satisfies a nonlinear non-degenerate parobolic equation for which we know, see for instance \cite[Section 7]{Evans}, that
$$\norm{\partial_t \mathbf{\tilde{n}}}_{L^2_tH^s_x} \leq C\norm{\mathbf{\tilde{n}^{\rm (in)}}}_{H^{s+1}_x}$$
and so we would obtain Theorem \ref{theo:BE} for $\mathbf{\tilde{n}^{\rm (in)}}$ in $H^{s+1}_x$ rather than $H^{s+2}_x$ but the solutions to the Boltzmann system would be weak in $L^2_t\mathcal{H}^s_\eps$.
\end{remark}
\bigskip
%
%

\section*{Acknowledgements}
The authors wish to thank Laurent Boudin for fruitful discussions on the Fick hydrodynamic limit of the Boltzmann equations for mixtures and the explicit expression of Fick diffusion coefficients.
\bibliographystyle{acm}
\bibliography{bibliography_BE_multi_to_Fick}

\begin{thebibliography}{10}

\bibitem{AAP}
{\sc Andries, P., Aoki, K., and Perthame, B.}
\newblock A consistent {BGK}-type model for gas mixtures.
\newblock {\em J. Statist. Phys. 106}, 5-6 (2002), 993--1018.

\bibitem{BBBD}
{\sc Baranger, C., Bisi, M., Brull, S., and Desvillettes, L.}
\newblock On the {C}hapman-{E}nskog asymptotics for a mixture of monoatomic and
  polyatomic rarefied gases.
\newblock In {\em AIP Conference Proceedings\/} (2019), vol.~2132, AIP
  Publishing, p.~020002.

\bibitem{BarMou}
{\sc Baranger, C., and Mouhot, C.}
\newblock Explicit spectral gap estimates for the linearized {B}oltzmann and
  {L}andau operators with hard potentials.
\newblock {\em Rev. Mat. Iberoamericana 21}, 3 (2005), 819--841.

\bibitem{BGL1}
{\sc Bardos, C., Golse, F., and Levermore, C.}
\newblock Fluid dynamic limits of kinetic equations. {I}. {F}ormal derivations.
\newblock {\em J. Statist. Phys. 63}, 1-2 (1991), 323--344.

\bibitem{BGL2}
{\sc Bardos, C., Golse, F., and Levermore, C.}
\newblock Fluid dynamic limits of kinetic equations. {II}. {C}onvergence proofs
  for the {B}oltzmann equation.
\newblock {\em Comm. Pure Appl. Math. 46}, 5 (1993), 667--753.

\bibitem{BD}
{\sc Bianca, C., and Dogbe, C.}
\newblock Recovering {N}avier-{S}tokes equations from asymptotic limits of the
  {B}oltzmann gas mixture equation.
\newblock {\em Commun. Theor. Phys. 65\/} (2016), 553--562.

\bibitem{BisDes}
{\sc Bisi, M., and Desvillettes, L.}
\newblock Formal passage from kinetic theory to incompressible
  {N}avier--{S}tokes equations for a mixture of gases.
\newblock {\em ESAIM: Mathematical Modelling and Numerical Analysis 48}, 4
  (2014), 1171--1197.

\bibitem{BisMarSpi}
{\sc Bisi, M., Martal\`o, G., and Spiga, G.}
\newblock Multi-temperature hydrodynamic limit from kinetic theory in a mixture
  of rarefied gases.
\newblock {\em Acta Appl. Math. 122\/} (2012), 37--51.

\bibitem{BonBri}
{\sc Bondesan, A., and Briant, M.}
\newblock Perturbative {C}auchy theory for a flux-incompressible
  {M}axwell-{S}tefan system in a non-equimolar regime.
\newblock Preprint, 2019.

\bibitem{Bothe}
{\sc Bothe, D.}
\newblock On the {M}axwell-{S}tefan approach to multicomponent diffusion.
\newblock In {\em Parabolic problems}, vol.~80 of {\em Progr. Nonlinear
  Differential Equations Appl.} Birkh\"auser/Springer Basel AG, Basel, 2011,
  pp.~81--93.

\bibitem{BGP1}
{\sc Boudin, L., Grec, B., and Pavan, V.}
\newblock The {M}axwell--{S}tefan diffusion limit for a kinetic model of
  mixtures with general cross sections.
\newblock {\em Nonlinear Analysis 159\/} (2017), 40--61.

\bibitem{BGP2}
{\sc Boudin, L., Grec, B., and Pavan, V.}
\newblock Diffusion models for mixtures using a stiff dissipative hyperbolic
  formalism.
\newblock {\em Journal of Hyperbolic Differential Equations 16}, 02 (2019),
  293--312.

\bibitem{BGPS}
{\sc Boudin, L., Grec, B., Pavi\'c, M., and Salvarani, F.}
\newblock Diffusion asymptotics of a kinetic model for gaseous mixtures.
\newblock {\em Kinetic and Related Models 6}, 1 (2013), 137--157.

\bibitem{BGS0}
{\sc Boudin, L., Grec, B., and Salvarani, F.}
\newblock A mathematical and numerical analysis of the {M}axwell-{S}tefan
  diffusion equations.
\newblock {\em Discrete \& Continuous Dynamical Systems - B 17\/} (2012), 1427.

\bibitem{BGS}
{\sc Boudin, L., Grec, B., and Salvarani, F.}
\newblock The {M}axwell-{S}tefan diffusion limit for a kinetic model of
  mixtures.
\newblock {\em Acta Applicandae Mathematicae 136}, 1 (2015), 79--90.

\bibitem{Briant}
{\sc Briant, M.}
\newblock From the {B}oltzmann equation to the incompressible {N}avier-{S}tokes
  equations on the torus: a quantitative error estimate.
\newblock {\em J. Differential Equations 259}, 11 (2015), 6072--6141.

\bibitem{BriDau}
{\sc Briant, M., and Daus, E.~S.}
\newblock The {B}oltzmann equation for a multi-species mixture close to global
  equilibrium.
\newblock {\em Arch. Ration. Mech. Anal. 222}, 3 (2016), 1367--1443.

\bibitem{BPS}
{\sc Brull, S., Pavan, V., and Schneider, J.}
\newblock Derivation of a {BGK} model for mixtures.
\newblock {\em Eur. J. Mech. B Fluids 33\/} (2012), 74--86.

\bibitem{Caflisch}
{\sc Caflisch, R.~E.}
\newblock The fluid dynamic limit of the nonlinear {B}oltzmann equation.
\newblock {\em Comm. Pure Appl. Math. 33}, 5 (1980), 651--666.

\bibitem{Ce}
{\sc Cercignani, C.}
\newblock {\em The {B}oltzmann equation and its applications}, vol.~67 of {\em
  Applied Mathematical Sciences}.
\newblock Springer-Verlag, New York, 1988.

\bibitem{CIP}
{\sc Cercignani, C., Illner, R., and Pulvirenti, M.}
\newblock {\em The mathematical theory of dilute gases}, vol.~106 of {\em
  Applied Mathematical Sciences}.
\newblock Springer-Verlag, New York, 1994.

\bibitem{DauJunMouZam}
{\sc Daus, E.~S., J{\"u}ngel, A., Mouhot, C., and Zamponi, N.}
\newblock Hypocoercivity for a linearized multispecies {B}oltzmann system.
\newblock {\em SIAM J. Math. Anal. 48}, 1 (2016), 538--568.

\bibitem{DMEL}
{\sc De~Masi, A., Esposito, R., and Lebowitz, J.~L.}
\newblock Incompressible {N}avier-{S}tokes and {E}uler limits of the
  {B}oltzmann equation.
\newblock {\em Comm. Pure Appl. Math. 42}, 8 (1989), 1189--1214.

\bibitem{DesLepMou}
{\sc Desvillettes, L., Lepoutre, T., and Moussa, A.}
\newblock Entropy, duality, and cross diffusion.
\newblock {\em SIAM Journal on Mathematical Analysis 46}, 1 (2014), 820--853.

\bibitem{DesLepMouTre}
{\sc Desvillettes, L., Lepoutre, T., Moussa, A., and Trescases, A.}
\newblock On the entropic structure of reaction-cross diffusion systems.
\newblock {\em Communications in Partial Differential Equations 40}, 9 (2015),
  1705--1747.

\bibitem{DMS}
{\sc Desvillettes, L., Monaco, R., and Salvarani, F.}
\newblock A kinetic model allowing to obtain the energy law of polytropic gases
  in the presence of chemical reactions.
\newblock {\em Eur. J. Mech. B Fluids 24}, 2 (2005), 219--236.

\bibitem{Evans}
{\sc Evans, L.~C.}
\newblock {\em Partial differential equations}, second~ed., vol.~19 of {\em
  Graduate Studies in Mathematics}.
\newblock American Mathematical Society, Providence, RI, 2010.

\bibitem{GSB}
{\sc Garz\'o, V., Santos, A., and Brey, J.~J.}
\newblock A kinetic model for a multicomponent gas.
\newblock {\em Phys. Fluids A 1}, 2 (1989), 380--383.

\bibitem{Gio}
{\sc Giovangigli, V.}
\newblock {\em Multicomponent flow modeling}.
\newblock Modeling and Simulation in Science, Engineering and Technology.
  Birkh\"auser Boston, Inc., Boston, MA, 1999.

\bibitem{GSR}
{\sc Golse, F., and Saint-Raymond, L.}
\newblock The {N}avier-{S}tokes limit of the {B}oltzmann equation for bounded
  collision kernels.
\newblock {\em Invent. Math. 155}, 1 (2004), 81--161.

\bibitem{Gr1}
{\sc Grad, H.}
\newblock Principles of the kinetic theory of gases.
\newblock In {\em Handbuch der {P}hysik (herausgegeben von {S}. {F}l\"ugge),
  {B}d. 12, {T}hermodynamik der {G}ase}. Springer-Verlag, Berlin, 1958,
  pp.~205--294.

\bibitem{Guo}
{\sc Guo, Y.}
\newblock Boltzmann diffusive limit beyond the {N}avier-{S}tokes approximation.
\newblock {\em Comm. Pure Appl. Math. 59}, 5 (2006), 626--687.

\bibitem{HutSal17}
{\sc Hutridurga, H., and Salvarani, F.}
\newblock Maxwell--stefan diffusion asymptotics for gas mixtures in
  non-isothermal setting.
\newblock {\em Nonlinear Analysis 159\/} (2017), 285--297.

\bibitem{HS1}
{\sc Hutridurga, H., and Salvarani, F.}
\newblock On the {M}axwell-{S}tefan diffusion limit for a mixture of monatomic
  gases.
\newblock {\em Math. Methods Appl. Sci. 40}, 3 (2017), 803--813.

\bibitem{JS}
{\sc J\"ungel, A., and Stelzer, I.~V.}
\newblock Existence analysis of {M}axwell-{S}tefan systems for multicomponent
  mixtures.
\newblock {\em SIAM J. Math. Anal. 45}, 4 (2013), 2421--2440.

\bibitem{Mou1}
{\sc Mouhot, C.}
\newblock Explicit coercivity estimates for the linearized {B}oltzmann and
  {L}andau operators.
\newblock {\em Comm. Partial Differential Equations 31}, 7-9 (2006),
  1321--1348.

\bibitem{MouNeu}
{\sc Mouhot, C., and Neumann, L.}
\newblock Quantitative perturbative study of convergence to equilibrium for
  collisional kinetic models in the torus.
\newblock {\em Nonlinearity 19}, 4 (2006), 969--998.

\bibitem{SR}
{\sc Saint-Raymond, L.}
\newblock {\em Hydrodynamic limits of the Boltzmann equation}.
\newblock No.~1971. Springer Science \& Business Media, 2009.

\bibitem{Vi2}
{\sc Villani, C.}
\newblock A review of mathematical topics in collisional kinetic theory.
\newblock In {\em Handbook of mathematical fluid dynamics, {V}ol. {I}}.
  North-Holland, Amsterdam, 2002, pp.~71--305.

\end{thebibliography}
%
%
\bigskip
\signmarc
\signberenice

\end{document}